\documentclass{amsart}
\usepackage{amsmath}
\usepackage{amsthm,amsfonts,amssymb,mathrsfs}
\usepackage{epic,eepic}
\usepackage{yfonts}
\usepackage{paralist,enumerate}
\usepackage[all]{xy}

\usepackage{xcolor,graphicx}
\usepackage[mathscr]{euscript}

\theoremstyle{definition}
\newtheorem{theorem}{Theorem}[section]
\newtheorem{definition}[theorem]{Definition}
\newtheorem{lemma}[theorem]{Lemma}
\newtheorem{proposition}[theorem]{Proposition}

\newtheorem*{theorem*}{Theorem}

\theoremstyle{remark}
\newtheorem{remark}[theorem]{Remark}

\begin{document}

\title[Tropical currents, extremality, and approximation]{A tropical approach to a generalized Hodge conjecture for positive currents}
\author{Farhad Babaee and June Huh}
\address{
Universit\'e de Fribourg\\
Chemin du Mus\'ee 23\\
CH-1700 Fribourg, Switzerland\\}
\email{farhad.babaee@unifr.ch}
\address{Princeton University and Institute for Advanced Study \\ Einstein Drive \\ Princeton \\  NJ 08540 \\ USA}
\email{huh@princeton.edu}
\thanks{This research was conducted during the period Farhad Babaee was a PhD student at the Universities of Bordeaux and Padova, and postdoc at the Concordia University, Ecole Normale Sup\'erieure of Paris, and Universit\'e de Fribourg. He was partially supported by the SNSF:PP00P2~150552$\slash$1 grant. June Huh was supported by a Clay Research Fellowship and NSF Grant DMS-1128155.}
\begin{abstract}

Demailly  showed that the Hodge conjecture is equivalent to the statement that any $(p,p)$-dimensional closed current with rational cohomology class can be approximated by linear combinations of integration currents. Moreover, he showed that the Hodge conjecture follows from the statement that all strongly positive closed currents with rational cohomology class can be approximated by positive linear combinations of integration currents \cite{DemaillyHodge}. In this article, we find a current which does not verify the latter statement on a smooth projective variety for which the Hodge conjecture is known to hold.  To construct this current, we extend the framework of `tropical currents' introduced in \cite{Babaee} from tori to toric varieties. We discuss  extremality properties of tropical currents, and show that the cohomology class of a tropical current is the recession of its underlying tropical variety. The counterexample is obtained from a tropical surface in $\mathbb{R}^4$ whose intersection form does not have the right signature in terms of the Hodge index theorem. 
\end{abstract}
\maketitle
\section{Introduction}

The main goal of this article is to construct an example that does not satisfy a strong version of the Hodge conjecture for strongly positive currents introduced in \cite{DemaillyHodge}. To state our main results, we first recall some basic definitions, following \cite[Chapter I]{DemaillyBook1}. 

Let $X$ be a complex manifold of dimension $n$. If $k$ is a nonnegative integer, we denote by $\mathscr{D}^{k}(X)$ the space of smooth complex differential forms of degree $k$ with compact support, endowed with the inductive limit topology. The space of currents of dimension $k$ is the topological dual space $\mathscr{D}'_{k}(X)$, that is, the space of all continuous linear functionals on $\mathscr{D}^{k}(X)$:
\[
\mathscr{D}'_{k}(X):=\mathscr{D}^{k}(X)'.
\]
The pairing between a current $\mathscr{T}$ and a differential form $\varphi$ will be denoted $\langle \mathscr{T},\varphi \rangle$. A $k$-dimensional current $\mathscr{T}$ is a \emph{weak limit} of a sequence of $k$-dimensional currents $\mathscr{T}_i$  if 
\[
\lim_{i \to \infty} \langle \mathscr{T}_i, \varphi \rangle = \langle \mathscr{T}, \varphi \rangle \ \ \text{for all $\varphi \in \mathscr{D}^{k}(X)$}. 
\]
There are corresponding decompositions according to the bidegree and bidimension
\[
\mathscr{D}^k(X)=\bigoplus_{p+q=k} \mathscr{D}^{p,q}(X), \quad \mathscr{D}'_k(X)=\bigoplus_{p+q=k} \mathscr{D}'_{p,q}(X).
\]
Most operations on smooth differential forms extend by duality to currents. For instance, the exterior derivative of a $k$-dimensional current $\mathscr{T}$ is the $(k-1)$-dimensional current  $d\mathscr{T}$ defined by
\[
\langle d\mathscr{T},\varphi \rangle = (-1)^{k+1} \langle \mathscr{T}, d\varphi \rangle,  \quad \varphi \in \mathscr{D}^{k-1}(X).
\]
The current $\mathscr{T}$ is \emph{closed} if its exterior derivative vanishes, and $\mathscr{T}$ is \emph{real} if it is invariant under the complex conjugation. When $\mathscr{T}$ is closed, it defines a cohomology class  of $X$, denoted $\{\mathscr{T}\}$.

The space of smooth differential forms of bidegree $(p,p)$ contains the cone of positive differential forms.
By definition, a smooth differential $(p,p)$-form $\varphi$ is \emph{positive} if
\[
\text{$\varphi(x)|_S$ is a nonnegative volume form for all $p$-planes $S \subseteq T_x X$ and $x \in X$}.
\]
Dually, a current $\mathscr{T}$ of bidimension $(p,p)$ is \emph{strongly positive}  if 
\[
\langle \mathscr{T}, \varphi \rangle \ge 0 \ \ \text{for every positive differential $(p,p)$-form $\varphi$ on $X$}.
\]
Integrating along complex analytic subsets of $X$ provides an important class of strongly positive currents on $X$.
If $Z$ is a $p$-dimensional complex analytic subset of $X$, then the \emph{integration current} $[Z]$ is the $(p,p)$-dimensional current defined by integrating over the smooth locus
\[
\big\langle [Z], \varphi \big\rangle = \int_{Z_{\text{reg}}} \varphi, \quad \varphi \in \mathscr{D}^{p,p}(X).
\]

Suppose from now on that $X$ is an $n$-dimensional smooth projective algebraic variety over the complex numbers, and let $p$ and $q$ be nonnegative integers with $p+q=n$. Let us consider the following statements:
\begin{enumerate}
\item[($\textrm{HC}$)] The Hodge conjecture: The intersection
\[
H^{2q}(X,\mathbb{Q})~ \cap H^{q,q}(X) 
\]
consists of classes of $p$-dimensional algebraic cycles with rational coefficients.\\
\item[($\textrm{HC}'$)] The Hodge conjecture for currents: If $\mathscr{T}$ is a $(p,p)$-dimensional real closed current on $X$ with cohomology class
\[
\{\mathscr{T}\} \in  \mathbb{R} \otimes_\mathbb{Z} \big(H^{2q}(X,\mathbb{Z})/\textrm{tors}~ \cap H^{q,q}(X) \big),
\]
then $\mathscr{T}$ is a weak limit of the form 
\[
\mathscr{T}=\lim_{i \to \infty} \mathscr{T}_i, \quad \mathscr{T}_i=\sum_j \lambda_{ij} [Z_{ij}],
\] 
where $\lambda_{ij}$ are real numbers and $Z_{ij}$ are $p$-dimensional subvarieties of $X$.\\
\item[($\textrm{HC}^+$)] The Hodge conjecture for strongly positive currents: 
If $\mathscr{T}$ is a $(p,p)$-dimensional strongly positive closed current on $X$ with cohomology class
\[
\{\mathscr{T}\} \in  \mathbb{R} \otimes_\mathbb{Z} \big(H^{2q}(X,\mathbb{Z})/\textrm{tors}~ \cap H^{q,q}(X) \big),
\]
then $\mathscr{T}$ is a weak limit of the form 
\[
\mathscr{T}=\lim_{i \to \infty} \mathscr{T}_i, \quad \mathscr{T}_i=\sum_j \lambda_{ij} [Z_{ij}],
\]
where $\lambda_{ij}$ are positive real numbers and $Z_{ij}$ are $p$-dimensional subvarieties of $X$.
\end{enumerate}
Demailly proved in  \cite[Th\'eor\`eme 1.10]{DemaillyHodge} that, for any  smooth projective variety and $q$ as above, 
\[
\textrm{HC}^+ \implies \textrm{HC}.
\]
Furthermore, he showed that $\textrm{HC}^+$ holds for any smooth projective variety when $q=1$, see \cite[Th\'eor\`eme 1.9]{DemaillyHodge} and the proof given in \cite[Chapter 13]{DemaillyBook2}. 
In \cite[Theorem 13.40]{DemaillyBook2}, Demailly showed that, in fact, for any smooth projective variety and $q$,
\[
\textrm{HC} \iff \textrm{HC}',
\]
and asked whether $\textrm{HC}'$ implies $\textrm{HC}^+$  \cite[Remark 13.43]{DemaillyBook2}.
In Theorem \ref{Main4}, we show that $\textrm{HC}^+$ fails even on toric varieties, where the Hodge conjecture readily holds: 

\begin{theorem}\label{Main4-intro}
There is a $4$-dimensional smooth projective toric variety $X$ and a $(2,2)$-dimensional strongly positive closed current $\mathscr{T}$ on $X$ with the following properties:
\begin{enumerate}[(1)]
\item The cohomology class of $\mathscr{T}$ satisfies
\[
\{\mathscr{T}\} \in H^4(X,\mathbb{Z})\slash\textrm{tors}~ \cap H^{2,2}(X).
\]
\item The current $\mathscr{T}$ is not a weak limit of the form
\[
\lim_{i \to \infty} \mathscr{T}_i, \quad \mathscr{T}_i=\sum_j \lambda_{ij} [Z_{ij}],
\]
where $\lambda_{ij}$ are nonnegative real numbers and $Z_{ij}$ are algebraic surfaces in $X$.
\end{enumerate}
\end{theorem}

The above current $\mathscr{T}$  generates an extremal ray of the cone of strongly positive closed currents on $X$: If $\mathscr{T}=\mathscr{T}_1+\mathscr{T}_2$ is any decomposition of $\mathscr{T}$ into strongly positive closed currents, then both $\mathscr{T}_1$ and $\mathscr{T}_2$ are nonnegative multiples of $\mathscr{T}$.
This extremality relates to $\textrm{HC}^+$ by the following application of Milman's converse to the Krein-Milman theorem, see Proposition \ref{Milman'sConverse} and
 \cite[Proof of Proposition 5.2]{DemaillyHodge}.

\begin{proposition}\label{Milman'sConverse-Intro}
Let $X$ be an algebraic variety and let $\mathscr{T}$ be a $(p,p)$-dimensional current on $X$ of the form
\[
\mathscr{T}=\lim_{i \to \infty} \mathscr{T}_i, \quad \mathscr{T}_i=\sum_j \lambda_{ij} [Z_{ij}],
\]
where $\lambda_{ij}$ are nonnegative real numbers and $Z_{ij}$ are $p$-dimensional irreducible subvarieties of $X$.
If $\mathscr{T}$ generates an extremal ray of the cone of strongly positive closed currents on $X$, then
 there are nonnegative real numbers $\lambda_i$ and $p$-dimensional irreducible subvarieties $Z_i \subseteq X$  such that
\[
\mathscr{T}=\lim_{i \to \infty} \lambda_i [Z_i].
\]
\end{proposition}
Therefore, if we assume that $\textrm{HC}^+$ holds for a smooth projective variety $X$, then every extremal strongly positive closed current with rational cohomology class can be approximated by positive multiples of integration currents along \emph{irreducible} subvarieties of $X$. 
Lelong in \cite{Lelong} proved that the integration currents along  irreducible analytic subsets are extremal, and asked whether those are the only extremal currents. Demailly in \cite{DemaillyHodge}  found the first extremal strongly positive closed current on $\mathbb{C}\mathbb{P}^2$ with a support of real dimension 3, which, therefore, cannot be an integration current along any analytic set. Later on, Bedford noticed that many extremal currents occur in dynamical systems on several complex variables have fractal sets as their support, and extremal currents of this type were later generated in several works such as~\cite{Bedford-Smillie, Forn-Sib, Bedford-Lyubich-Smillie, sibony, Cantat, Diller-Favre, Guedj02,  Guedj-Sibony} in codimension $1$, and \cite{Dinh-Sibony1, Guedj, Dinh-Sibony3, sib-dinh-rigidity} in higher codimension. These extremal currents, though, were readily known to be a weak limit of integration currents by the methods of their construction. The first tropical approach to extremal currents was established in the PhD thesis of the first author \cite{Babaee}. He introduced the notion of tropical currents and deduced certain sufficient local conditions which implied extremality in any dimension and codimension.

In Section \ref{SectionConstruction}, we provide a detailed construction of tropical currents. A \emph{tropical current} is a certain closed current of bidimension $(p,p)$ on the algebraic torus $(\mathbb{C}^*)^n$, which is associated to a tropical variety of dimension $p$ in $\mathbb{R}^n$. A tropical variety is a weighted rational polyhedral complex $\mathscr C$ which is \emph{balanced}, see Definition \ref{BalancingCondition}. The tropical current associated to $\mathscr C$, denoted by $\mathscr{T}_\mathscr{C}$, has support 
\[
|\mathscr{T}_\mathscr{C}|= \textrm{Log}^{-1}(\mathscr C), 
\]
where $\textrm{Log}$ is the map defined by
\[
\textrm{Log}: (\mathbb{C}^*)^n \longrightarrow \mathbb{R}^n, \quad \big(z_1,\ldots,z_n) \longmapsto (-{\log|z_1|},\ldots, -{\log|z_n|}\big).
\]
To construct $\mathscr{T}_\mathscr{C}$ from a weighted complex $\mathscr{C}$, for each $p$-dimensional cell $\sigma$ in $\mathscr{C}$ we consider a current $\mathscr{T}_\sigma$, the average of the integration currents along fibers of a natural fiberation over the real torus $\textrm{Log}^{-1}(\sigma) \longrightarrow (S^1)^{n-p}$. 
The current $\mathscr{T}_{\mathscr C}$ is then defined by setting
\[
\mathscr{T}_{\mathscr C} = \sum_{\sigma} \text{w}_{\mathscr{C}}(\sigma) \mathscr{T}_\sigma, 
\]
where the sum is over all $p$-dimensional cells in $\mathscr C$ and $\text{w}_{\mathscr{C}}(\sigma)$ is the corresponding weight. 
In Theorem \ref{Main1}, we give the following criterion for the closedness of the resulting current $\mathscr{T}_\mathscr{C}$, cf.  \cite[Theorem 3.1.8]{Babaee}.

\begin{theorem}\label{Main1-Intro}
A weighted complex $\mathscr C$ is balanced if and only if the current $\mathscr{T}_\mathscr{C}$ is closed.
\end{theorem}

In Section \ref{SectionFourier}, we prove the above criterion for closedness of $\mathscr{T}_\mathscr{C}$, as well as the following criterion for strong extremality of $\mathscr{T}_\mathscr{C}$.
A closed current $\mathscr{T}$ with measure coefficients is said to be \emph{strongly extremal} if any closed current $\mathscr{T}'$ with measure coefficients which has the same dimension and  support as $\mathscr{T}$ is a constant multiple of $\mathscr{T}$.
(Note that if $\mathscr{T}$ is strongly positive and strongly extremal, then  $\mathscr{T}$ generates an extremal ray in the cone of strongly positive closed currents.)
Similarly, a balanced weighted complex $\mathscr{C}$ is said to be \emph{strongly extremal} if any balanced weighted complex $\mathscr{C}'$ which has the same dimension and  support as $\mathscr{C}$ is a constant multiple of $\mathscr{C}$.
In Theorem \ref{Main2}, we prove the following  improvement of extremality results in \cite{Babaee}.

\begin{theorem}\label{Main2-Intro}
A non-degenerate tropical variety $\mathscr{C}$ is strongly extremal if and only if the tropical current $\mathscr{T}_\mathscr{C}$ is  strongly extremal.
\end{theorem}

Here a tropical variety in $\mathbb{R}^n$ is said to be \emph{non-degenerate} if its support is contained in no proper subspace of $\mathbb{R}^n$.
We note that there is an abundance of strongly extremal tropical varieties. For example, the Bergman fan of any simple matroid is a strongly extremal tropical variety \cite[Theorem 38]{Huh}.
There are $376467$ nonisomorphic simple matroids on $9$ elements \cite{Mayhew-Royle}, producing that many strongly extremal strongly positive closed currents on $(\mathbb{C}^*)^8$. By Theorem \ref{Main3-Intro} below, all of them have distinct cohomology classes in one fixed toric compactification of $(\mathbb{C}^*)^8$, the one associated to the permutohedron \cite{Huh}. In fact, Demailly's first example of a non-analytic extremal strongly positive current in \cite{DemaillyHodge} is the tropical current associated to the simplest nontrivial matroid, namely the rank $2$ simple matroid on $3$ elements.

In Section \ref{SectionToric}, we consider the trivial extension $\overline{\mathscr{T}}_\mathscr{C}$ of the tropical current $\mathscr{T}_\mathscr{C}$ to an $n$-dimensional smooth projective toric variety $X$ whose fan is \emph{compatible} with $\mathscr{C}$, see Definition \ref{DefinitionCompatible}.
According to Fulton and Sturmfels \cite{Fulton-Sturmfels}, cohomology classes of a complete toric variety bijectively correspond to balanced weighted fans compatible with the fan of the toric variety. In Theorem \ref{Main3}, we give a complete description of the cohomology class of $\overline{\mathscr{T}}_\mathscr{C}$ in $X$:

\begin{theorem}\label{Main3-Intro}
If $\mathscr{C}$ is a $p$-dimensional tropical variety compatible with the fan of $X$, then
\[
\{\overline{\mathscr{T}}_\mathscr{C}\}=\text{rec}(\mathscr{C}) \in H^{q,q}(X),
\]
where $\textrm{rec}(\mathscr C)$ is the \emph{recession} of $\mathscr C$ (recalled  in Section \ref{SubsectionRecession}).
In particular, if all polyhedrons in $\mathscr{C}$ are cones in $\Sigma$, then
\[
\{\overline{\mathscr{T}}_\mathscr{C}\}=\mathscr{C} \in H^{q,q}(X).
\]
\end{theorem}

The current $\mathscr{T}$ in Theorem \ref{Main4-intro} is a current of the form $\overline{\mathscr{T}}_\mathscr{C}$, and Theorem \ref{Main3-Intro} plays an important role in justifying the claimed properties of $\mathscr{T}$.

In Section \ref{SectionLast}, we complete the proof of Theorem \ref{Main4-intro} by analyzing a certain Laplacian matrix associated to a $2$-dimensional tropical variety $\mathscr{C}$. According to Theorem \ref{Main3-Intro}, if $\mathscr{C}$ is compatible with the fan of an $n$-dimensional smooth projective toric variety $X$, we may view the cohomology class of  $\overline{\mathscr{T}}_\mathscr{C}$ as a geometric graph $G=G(\mathscr{C}) \subseteq \mathbb{R}^n \setminus \{0\}$ with edge weights $w_{ij}$ satisfying the \emph{balancing condition}: At each vertex $u_i$ there is a real number $d_i$ such that
\[
d_iu_i=\sum_{u_i \sim u_j} w_{ij} u_j,
\]
where the sum is over all neighbors of $u_i$ in $G$.  
We define the \emph{tropical Laplacian} of $\mathscr{C}$ to be the real symmetric matrix
$L_{G}$ with entries
\[
(L_{G})_{ij}:=\left\{\begin{array}{cl} d_i & \text{if $u_i=u_j$,} \\ -w_{ij} & \text{if $u_i \sim u_j$,} \\ 0 & \text{if otherwise,}\end{array}\right.
\]
where the diagonal entries $d_i$ are the real numbers satisfying
\[
d_i u_i = \sum_{u_i \sim u_j} w_{ij} u_j.
\]
When $G$ is the graph of a polytope with weights given by the Hessian of the volume of the dual polytope, the matrix $L_G$ has been considered in various contexts related to rigidity and polyhedral combinatorics \cite{Connelly,Filliman,Lovasz,Izmestiev}. In this case, $L_G$ is known to have exactly one negative eigenvalue, by the Alexandrov-Fenchel inequality. See, for example, \cite[Proposition 4]{Filliman} and \cite[Theorem A.10]{Izmestiev}.
In Proposition \ref{OneNegativeEigenvalue}, using the Hodge index theorem and the continuity of the cohomology class assignment, we show that $L_G$ has at most one negative eigenvalue if $\overline{\mathscr{T}}_{\mathscr{C}}$ is a weak limit of integration currents along irreducible surfaces in $X$.

The remainder of the paper is devoted to the construction of a strongly extremal tropical surface $\mathscr{C}$ whose tropical Laplacian has more than one negative eigenvalue.
For this we introduce two operations on weighted fans, $F \longmapsto F_{ij}^+$ (Section \ref{PlusConstruction}) and $F \longmapsto F_{ij}^-$ (Section \ref{MinusConstruction}), and repeatedly apply them to a geometric realization of the complete bipartite graph $K_{4,4} \subseteq \mathbb{R}^4$ to arrive at $\mathscr{C}$ with the desired properties. By the above Theorems \ref{Main1-Intro}, \ref{Main2-Intro}, and \ref{Main3-Intro}, the resulting tropical current $\overline{\mathscr{T}}_\mathscr{C}$ is a strongly extremal strongly positive closed current which is not a weak limit of positive linear combinations of integration currents along subvarieties.

\subsection*{Acknowledgements}
We thank the referees for their careful reading and insightful comments. 
Their suggestions significantly improved the quality of the manuscript. 
The first author is also thankful to Alain Yger, Jean-Pierre Demailly, Vincent Koziarz, Erwan Brugall\'e, Alexander Rashkovskii, Romain Dujardin, Omid Amini, and Charles Favre for the fruitful discussions, and their support.

\section{Construction of tropical currents}\label{SectionConstruction}

\subsection{}

Let $\mathbb{C}^*$ be the group of nonzero complex numbers.
The \emph{logarithm map} is the homomorphism
\[
-\text{log}: \mathbb{C}^* \longrightarrow \mathbb{R}, \qquad z \longmapsto -\log |z|,
\]
and the \emph{argument map} is the homomorphism
\[
\text{arg}: \mathbb{C}^* \longrightarrow S^1, \qquad z \longmapsto z/|z|.
\]
The argument map splits the exact sequence
\[
\xymatrix{
0\ar[r]&S^1\ar[r]&\mathbb{C}^*\ar[r]^{-\text{log}}&\mathbb{R}\ar[r]&0,
}
\]
giving polar coordinates to nonzero complex numbers.
Under the chosen sign convention, the inverse image of $\mathbb{R}_{> 0}$ under the logarithm map is the punctured unit disk
\[
\mathbb{D}^*:=\{z \in \mathbb{C}^* \mid |z| <1\}.
\]

Let $N$ be a finitely generated free abelian group. 
There are Lie group homomorphisms
\[
\xymatrix{
&T_N \ar[dr]^{\text{arg}\otimes_\mathbb{Z}1} \ar[dl]_{-\text{log}\otimes_\mathbb{Z}1}&\\
N_\mathbb{R}&&S_N,}
\]
called the \emph{logarithm map} and the \emph{argument map} for $N$ respectively, where 
\begin{eqnarray*}
T_N&:=&\text{the complex algebraic torus $\mathbb{C}^* \otimes_\mathbb{Z} N$,}\\
S_N&:=&\text{the compact real torus $S^1 \otimes_\mathbb{Z} N$,}\\
N_\mathbb{R}&:=&\text{the real vector space $\mathbb{R} \otimes_\mathbb{Z} N$.}
\end{eqnarray*}
When $N$ is the group $\mathbb{Z}^n$ of integral points in $\mathbb{R}^n$, we denote the two maps by
\[
\xymatrix{
&(\mathbb{C}^*)^n \ar[dr]^{\text{Arg}} \ar[dl]_{\text{Log}}&\\
\mathbb{R}^n&&(S^1)^n.
}
\]

\subsection{}

A linear subspace of $\mathbb{R}^n$ is \emph{rational} if it is generated by a subset of $\mathbb{Z}^n$. 
Corresponding to a $p$-dimensional rational subspace $H \subseteq \mathbb{R}^n$, there is a commutative diagram of split exact sequences
\[
 \xymatrix{
  &0 \ar[d]&0 \ar[d]&0 \ar[d]&\\
 0 \ar[r]& S_{H \cap \mathbb{Z}^n} \ar[r] \ar[d]& (S^1)^n \ar[r] \ar@/_/[d]&S_{\mathbb{Z}^n/(H \cap \mathbb{Z}^n)} \ar[r] \ar[d]& 0\\
 0 \ar[r]& T_{H \cap \mathbb{Z}^n} \ar[r] \ar[d]& (\mathbb{C}^*)^n \ar[r] \ar[d]^{\text{Log}} \ar@/_/[u]_{\text{Arg}}&T_{\mathbb{Z}^n/(H \cap \mathbb{Z}^n)} \ar[r] \ar[d]& 0\\
 0 \ar[r]& H\ar[r]\ar[d] & \mathbb{R}^n \ar[r]\ar[d] &\mathbb{R}^n/H \ar[r] \ar[d]& 0,\\
  &0&0&0&\\
}
\]
where the vertical surjections are the logarithm maps for $H \cap \mathbb{Z}^n$, $\mathbb{Z}^n$, and their quotient.
We define a Lie group homomorphism $\pi_H$ as the composition
\[
\pi_H:\xymatrix{ \text{Log}^{-1}(H) \ar[r]^{\quad\text{Arg}}& (S^1)^n \ar[r]&  S_{\mathbb{Z}^n/(H \cap \mathbb{Z}^n)}}.
\]
The map $\pi_H$ is a submersion, equivariant with respect to the action of $(S^1)^n$. Its kernel is the closed subgroup
\[
\text{ker}(\pi_H)=T_{H \cap \mathbb{Z}^n} \subseteq (\mathbb{C}^*)^n.
\]
Each fiber of $\pi_H$ is a translation of the kernel by the action of $(S^1)^n$,  
and in particular, each fiber $\pi_H^{-1}(x)$ is a $p$-dimensional closed complex submanifold of $(\mathbb{C}^*)^n$.

\begin{definition}
Let $\mu$ be a complex Borel measure on $S_{\mathbb{Z}^n/(H \cap \mathbb{Z}^n)}$.
We define a $(p,p)$-dimensional closed current $\mathscr{T}_H(\mu)$ on $(\mathbb{C}^*)^n$ by
\[
\mathscr{T}_H(\mu):=\int_{x \in S_{\mathbb{Z}^n/(H \cap \mathbb{Z}^n)}} \big[\pi_H^{-1}(x)\big] \ d\mu(x).
\]
When $\mu$ is the Haar measure on $S_{\mathbb{Z}^n/(H \cap \mathbb{Z}^n)}$ normalized by
\[
\int_{x \in S_{\mathbb{Z}^n/(H \cap \mathbb{Z}^n)}} d\mu(x)=1,
\] 
we omit $\mu$ from the notation and write
\[
\mathscr{T}_H:=\mathscr{T}_H(\mu).
\]
\end{definition}

In other words, $\mathscr{T}_H(\mu)$ is obtained from the $0$-dimensional current $d\mu$ by
\[
\mathscr{T}_H(\mu) = \iota^H_{*} \big( \pi_H^* (d\mu) \big),
\]
where  $\iota^H$ is the closed embedding and $\pi_H$ is the oriented submersion in the diagram
\[
\xymatrix{
\text{Log}^{-1}(H) \ar[r]^{\quad \iota^H} \ar[d]_{\pi_H}& (\mathbb{C}^*)^n\\
S_{\mathbb{Z}^n/(H \cap \mathbb{Z}^n)}. &
}
\] 
Each fiber of $\pi_H$ is invariant under the action of $T_{H \cap \mathbb{Z}^n}$, and hence the current $\mathscr{T}_H(\mu)$ remains invariant under the action of $T_{H\cap \mathbb{Z}^n}$:
\[
\mathscr{T}_H(\mu)=t_* \big(\mathscr{T}_H(\mu)\big) = t^* \big(\mathscr{T}_H(\mu)\big), \qquad t \in T_{H \cap \mathbb{Z}^n}.
\]
The current $\mathscr{T}_H(\mu)$ is strongly positive if and only if $\mu$ is a positive measure.

\subsection{}

Let $A$ be a $p$-dimensional affine subspace of $\mathbb{R}^n$ parallel to the linear subspace $H$.
For $a \in A$,  there is a commutative diagram of corresponding translations
\[
\xymatrix{
\text{Log}^{-1}(A) \ar[d]_{\text{Log}} \ar[r]^{e^a}& \text{Log}^{-1}(H) \ar[d]^{\text{Log}}\\
A \ar[r]^{-a}& H.
}
\]
We define a submersion $\pi_{A}$ as the composition
\[
\pi_{A}:\xymatrix{\text{Log}^{-1}(A) \ar[r]^{e^a \ } & \text{Log}^{-1}(H) \ar[r]^{\pi_{H} \ \ }&S_{\mathbb{Z}^n/(H \cap \mathbb{Z}^n)}.}
\]
The map $\pi_A$ does not depend on the choice of $a$, and 
each fiber of $\pi_A$ is a  $p$-dimensional closed complex submanifold of $(\mathbb{C}^*)^n$ invariant under the action of $T_{H \cap \mathbb{Z}^n}$.

\begin{definition}
Let $\mu$ be a complex Borel measure on $S_{\mathbb{Z}^n/H\cap \mathbb{Z}^n}$.
We define a $(p,p)$-dimensional closed current $\mathscr{T}_{A}(\mu)$ on $(\mathbb{C}^*)^n$ by
\[
\mathscr{T}_A(\mu):=\int_{x \in S_{\mathbb{Z}^n/(H \cap \mathbb{Z}^n)}} \big[\pi_A^{-1}(x)\big] \ d\mu(x).
\]
When $\mu$ is the normalized Haar measure on $S_{\mathbb{Z}^n/(H \cap \mathbb{Z}^n)}$,
we write
\[
\mathscr{T}_A:=\mathscr{T}_A(\mu).
\]
\end{definition}

The current $\mathscr{T}_A(\mu)$ is strongly positive if and only if $\mu$ is a positive measure,
and the construction is equivariant with respect to the action of $\mathbb{R}^n$ by translations:
\[
\mathscr{T}_{A-b}(\mu)=(e^{-b})^*\big(\mathscr{T}_{A}(\mu)\big), \qquad b \in \mathbb{R}^n.
\]

Note that $\mathscr{T}_A(\mu)$ has measure coefficients: For each open subset $U \subseteq (\mathbb{C}^*)^n$, the restriction of  $\mathscr{T}_A(\mu)|_U$ can be written in a unique way 
\[
\mathscr{T}_A(\mu)|_U=\sum_{|I|=|J|=n-p} \mu_{IJ} \ dz_I \wedge d\bar z_J,
\]
where $z_1,\ldots,z_n$ are coordinate functions and $\mu_{IJ}$ are complex Borel measures on $U$.
This expression can be used to define the current $\mathbf{1}_B \mathscr{T}_A(\mu)$, where $\mathbf{1}_B$ is the characteristic function of a Borel subset $B \subseteq (\mathbb{C}^*)^n$. We cover the torus by relatively compact open subsets $U \subseteq (\mathbb{C}^*)^n$, and set
\[
\mathbf{1}_B\mathscr{T}_A(\mu)|_U:=\sum_{|I|=|J|=n-p} \mu_{IJ}|_B \ dz_I \wedge d\bar z_J.
\]

\subsection{}

A \emph{rational polyhedron} in $\mathbb{R}^n$ is an intersection of finitely many half-spaces of the form
\[
\langle u,m\rangle \ge c, \quad m \in (\mathbb{Z}^n)^\vee, \quad c \in \mathbb{R}.
\]
Let $\sigma$ be a $p$-dimensional rational polyhedron in $\mathbb{R}^n$. We define
\begin{eqnarray*}
\text{aff}(\sigma)&:=& \text{the affine span of $\sigma$},\\
\sigma^\circ \ &:=& \text{the interior of $\sigma$ in $\text{aff}(\sigma)$},\\
H_\sigma \ &:=& \text{the linear subspace parallel to $\text{aff}(\sigma)$}.
\end{eqnarray*}
The \emph{normal lattice} of $\sigma$ is the quotient group
\[
N(\sigma):=\mathbb{Z}^n/(H_\sigma \cap \mathbb{Z}^n).
\]
The normal lattice defines the $(n-p)$-dimensional vector spaces
\[
N(\sigma)_\mathbb{R}:=\mathbb{R} \otimes_\mathbb{Z} N(\sigma), \qquad N(\sigma)_\mathbb{C}:=\mathbb{C} \otimes_\mathbb{Z} N(\sigma).
\]

\begin{definition}
Let $\mu$ be a complex Borel measure on $S_{N(\sigma)}$.
\begin{enumerate}[(1)]
\item We define a submersion $\pi_\sigma$ as the restriction of $\pi_{\text{aff}(\sigma)}$ to $\text{Log}^{-1}(\sigma^\circ)$:
\[
\pi_\sigma: \text{Log}^{-1}(\sigma^\circ) \longrightarrow S_{N(\sigma)}.
\]
\item We define a $(p,p)$-dimensional current $\mathscr{T}_\sigma(\mu)$ on $(\mathbb{C}^*)^n$ by
\[
\mathscr{T}_\sigma(\mu):=\mathbf{1}_{\text{Log}^{-1}(\sigma)} \mathscr{T}_{\text{aff}(\sigma)}(\mu).
\]
\end{enumerate}
When $\mu$ is the normalized Haar measure on $S_{N(\sigma)}$, we write
\[
\mathscr{T}_\sigma:=\mathscr{T}_\sigma(\mu).
\]
\end{definition}

Each fiber $\pi_\sigma^{-1}(x)$ is a $p$-dimensional complex manifold, being an open subset of the $p$-dimensional closed complex submanifold $\pi_{\text{aff}(\sigma)}^{-1}(x) \subseteq (\mathbb{C}^*)^n$.
The closure $\overline{\pi^{-1}_\sigma(x)}$ is a manifold with piecewise smooth boundary, and
\[
\mathscr{T}_\sigma(\mu)= \int_{x \in S_{N(\sigma)}} \big[\overline{\pi^{-1}_\sigma(x)}\big]d\mu(x).
\]
In other words, $\mathscr{T}_\sigma(\mu)$ is the trivial extension to $(\mathbb{C}^*)^n$ of the pullback of the $0$-dimensional current $d\mu$ along the oriented submersion $\pi_\sigma$. We compute the boundary of $\mathscr{T}_\sigma(\mu)$ in Proposition \ref{BoundaryPolyhedron} below.

The construction is equivariant with respect to the action of $\mathbb{R}^n$ by translations:
\[
\mathscr{T}_{\sigma-b}(\mu)=(e^{-b})^*\big(\mathscr{T}_{\sigma}(\mu)\big), \qquad b \in \mathbb{R}^n.
\]
The current $\mathscr{T}_\sigma(\mu)$ is strongly positive if and only if the measure $\mu$ is positive, and
its support satisfies
\[
|\mathscr{T}_\sigma(\mu)| \subseteq |\mathscr{T}_\sigma|=\text{Log}^{-1}(\sigma) \subseteq (\mathbb{C}^*)^n.
\]

\subsection{}

A polyhedral complex in $\mathbb{R}^n$ is \emph{locally finite} if any compact subset of $\mathbb{R}^n$ intersects only finitely many cells. It is easy to see that the construction of $\mathscr{T}_{\sigma}(\mu)$ behaves well with respect to subdivisions: 

\begin{proposition}\label{ConeCurrents}
If a $p$-dimensional rational polyhedron $\sigma$ is a union of $p$-dimensional rational polyhedrons $\sigma_i$ in a locally finite polyhedral complex,
then
\[
\mathscr{T}_\sigma(\mu)=\sum_{i} \mathscr{T}_{\sigma_i}(\mu).
\]
\end{proposition}

The sum is well-defined because the subdivision of $\sigma$ is locally finite.

The boundary of $\mathscr{T}_\sigma(\mu)$ has measure coefficients, and can be understood geometrically from the restrictions of the logarithm map for $\mathbb{Z}^n$ to fibers of $\pi_{\text{aff}(\sigma)}$:
\[
l_{\sigma,x}:\pi^{-1}_{\text{aff}(\sigma)}(x) \longrightarrow \text{aff}(\sigma), \qquad x \in S_{N(\sigma)}.
\]
Each $l_{\sigma,x}$ is a translation of the logarithm map for $H_\sigma \cap \mathbb{Z}^n$, and hence is a  submersion. 
We have
\[
\pi_\sigma^{-1}(x)=l_{\sigma,x}^{-1}(\sigma^\circ). 
\]
Since $l_{\sigma,x}$ is a submersion, the closure of the inverse image is the inverse image of the closure.
In particular, the closure of $l_{\sigma,x}(\sigma^\circ)$ in the ambient torus is  $l_{\sigma,x}^{-1}(\sigma)$. The closure has
the piecewise smooth boundary
\[
\partial \Big( l^{-1}_{\sigma,x}(\sigma)\Big)=\bigcup_\tau \ l_{\sigma,x}^{-1}(\tau),
\]
where the union is over all codimension $1$ faces $\tau$ of $\sigma$. The smooth locus of the boundary is the disjoint union
\[
\coprod_\tau \ l_{\sigma,x}^{-1}(\tau^\circ). 
\]
The complex manifold $l_{\sigma,x}(\sigma^\circ)$ has a canonical orientation, and it induces an orientation on each of its boundary components $l_{\sigma,x}^{-1}(\tau^\circ)$, each with  a real codimension $1$. 

\begin{proposition}\label{BoundaryPolyhedron}
For any complex Borel measure $\mu$ on $S_{N(\sigma)}$,
\[
d\mathscr{T}_\sigma(\mu)=-\sum_{\tau \subset \sigma} \Bigg( \int_{x \in S_{N(\sigma)}}    \big[l_{\sigma,x}^{-1}(\tau)\big] d\mu(x) \Bigg),
\]
where  the sum is over all codimension $1$ faces $\tau$ of $\sigma$. 
\end{proposition}

It follows that the support of $d\mathscr{T}_\sigma(\mu)$ satisfies
\[
 |d\mathscr{T}_\sigma(\mu) |  \subseteq |d\mathscr{T}_\sigma | = \bigcup_{\tau\subset \sigma} \ \text{Log}^{-1}(\tau) \subseteq (\mathbb{C}^*)^n.
\]

\begin{proof}
Subdividing $\sigma$ if necessary, we may assume that $\sigma$ is a manifold with corners.
By Stokes' theorem, 
\[
d\big[l_{\sigma,x}^{-1}(\sigma)\big]=-\sum_{\tau\subset \sigma} \big[l_{\sigma,x}^{-1}(\tau)\big], \qquad x \in S_{N(\sigma)}.
\]
Since $\pi_\sigma^{-1}(x)=l_{\sigma,x}^{-1}(\sigma^\circ)$, we have
\[
d\mathscr{T}_\sigma(\mu)= \int_{x \in S_{N(\sigma)}} d\big[\overline{\pi^{-1}_\sigma(x)}\big]d\mu(x)=-\sum_{\tau\subset \sigma} \Bigg( \int_{x \in S_{N(\sigma)}}    \big[l_{\sigma,x}^{-1}(\tau)\big] d\mu(x) \Bigg).
\]
\end{proof}

We consider the important special case when $\sigma$ is a $p$-dimensional \emph{unimodular cone} in $\mathbb{R}^n$, that is, a cone  generated by part of a lattice basis $u_1,\ldots,u_p$ of $\mathbb{Z}^n$.
Let $\tilde x$ be an element of $(S^1)^n$, and consider the closed embedding given by the monomial map 
\[
(\mathbb{C}^*)^p \longrightarrow (\mathbb{C}^*)^n, \qquad z \longmapsto \tilde x \cdot z^{[u_1,\ldots,u_p]},
\]
where $[u_1,\ldots,u_p]$ is the matrix with column vectors $u_1,\ldots,u_p$.
If $x$ is the image of $\tilde x$ in $S_{N(\sigma)}$ and $\tau$ is the cone generated by $u_2,\ldots,u_p$, then the map restricts to diffeomorphisms
\begin{eqnarray*}
\mathbb{C}^*\times (\mathbb{C}^*)^{p-1}  &\simeq& \pi^{-1}_{\text{aff}(\sigma)}(x),\\
 \mathbb{D}^* \times (\mathbb{D}^*)^{p-1} &\simeq&l_{\sigma,x}^{-1}(\sigma^\circ),\\
S^1 \times (\mathbb{D}^*)^{p-1} &\simeq& l^{-1}_{\sigma,x}(\tau^\circ).
\end{eqnarray*}

\subsection{}

A $p$-dimensional \emph{weighted complex} in $\mathbb{R}^n$ is a locally finite polyhedral complex $\mathscr{C}$ such that
\begin{enumerate}[(1)]
\item each inclusion-maximal cell $\sigma$ in $\mathscr{C}$ is rational, 
\item each inclusion-maximal cell $\sigma$ in $\mathscr{C}$ is $p$-dimensional, and
\item each  inclusion-maximal cell $\sigma$ in $\mathscr{C}$ is assigned a complex number $\text{w}_\mathscr{C}(\sigma)$.
\end{enumerate}
The weighted complex $\mathscr{C}$ is said to be \emph{positive} if,
for all $p$-dimensional cells $\sigma$ in $\mathscr{C}$,
\[
\text{w}_\mathscr{C}(\sigma) \ge 0.
\]
The \emph{support} $|\mathscr{C}|$ of $\mathscr{C}$ is the union of all $p$-dimensional cells of $\mathscr{C}$ with nonzero weight.

\begin{definition}
A $p$-dimensional weighted complex $\mathscr{C}'$ is a \emph{refinement} of $\mathscr{C}$ if 
$|\mathscr{C}'|=|\mathscr{C}|$
and each $p$-dimensional cell $\sigma' \in \mathscr{C}'$ with nonzero weight is contained in some $p$-dimensional cell $\sigma \in \mathscr{C}$ with 
\[
\text{w}_{\mathscr{C}'}(\sigma')=\text{w}_{\mathscr{C}}(\sigma).
\]
If $p$-dimensional weighted complexes $\mathscr{C}_1$ and $\mathscr{C}_2$ have a common refinement, we write
\[
\mathscr{C}_1 \sim \mathscr{C}_2.
\]
This defines an equivalence relation on the set of $p$-dimensional weighted complexes in $\mathbb{R}^n$.
\end{definition}

Note that any two $p$-dimensional weighted complexes in $\mathbb{R}^n$ can be added after suitable refinements of each. 
This gives the set of equivalence classes of $p$-dimensional weighted complexes in $\mathbb{R}^n$ the structure of a complex vector space.

\begin{definition}
We define a $(p,p)$-dimensional current $\mathscr{T}_\mathscr{C}$ on $(\mathbb{C}^*)^n$ by
\[
\mathscr{T}_\mathscr{C}:=\sum_{\sigma} \text{w}_\mathscr{C}(\sigma)\ \mathscr{T}_\sigma,
\]
where the sum is over all $p$-dimensional cells in $\mathscr{C}$. 
\end{definition}

For an explicit construction of $\mathscr{T}_\mathscr{C}$ involving coordinates, see \cite{Babaee}.
If $\mathscr{C}-b$ is the weighted complex obtained by translating $\mathscr{C}$ by $b \in \mathbb{R}^n$, then
\[
\mathscr{T}_{\mathscr{C}-b}=(e^{-b})^*(\mathscr{T}_\mathscr{C}). 
\]
The current $\mathscr{T}_\mathscr{C}$ is strongly positive if and only if the weighted complex $\mathscr{C}$ is positive. The support of $\mathscr{T}_\mathscr{C}$ is the closed subset 
\[
|\mathscr{T}_\mathscr{C}|=\text{Log}^{-1}|\mathscr{C}| \subseteq (\mathbb{C}^*)^n.
\]

Proposition \ref{ConeCurrents} implies that equivalent weighted complexes define the same current, and hence there is a map
from the set of equivalence classes of weighted complexes
\[
\varphi: \{\mathscr{C}\} \longmapsto \mathscr{T}_\mathscr{C}.
\]
For $p$-dimensional weighted complexes $\mathscr{C}_1,\mathscr{C}_2$ and complex numbers $c_1,c_2$, we have
\[
\mathscr{T}_{c_1  \{\mathscr{C}_1\}+c_2 \{\mathscr{C}_2\}}  =c_1 \mathscr{T}_{\{\mathscr{C}_1\}}+c_2  \mathscr{T}_{\{\mathscr{C}_2\}}.
\]
It is clear that the kernel of the linear map $\varphi$ is trivial, and hence 
\[
\mathscr{C}_1 \sim \mathscr{C}_2 \quad \text{if and only if} \quad \mathscr{T}_{\mathscr{C}_1}=\mathscr{T}_{\mathscr{C}_2}.
\]

\subsection{}

Let $\tau$ be a codimension $1$ face of a $p$-dimensional rational polyhedron $\sigma$.
The difference of $\sigma$ and $\tau$ generates a $p$-dimensional rational polyhedral cone containing $H_\tau$, defining a ray
in the normal space
\[
\text{cone}(\sigma-\tau)/H_\tau \subseteq H_\sigma/H_\tau \subseteq \mathbb{R}^n/H_\tau = N(\tau)_\mathbb{R}.
\]
We write $u_{\sigma/\tau}$ for the primitive generator of this ray in the lattice $N(\tau)$. For any $ b \in \mathbb{R}^n$,
\[
u_{\sigma-b/\tau-b}=u_{\sigma/\tau}.
\]

\begin{definition}\label{BalancingCondition}
A $p$-dimensional weighted complex $\mathscr{C}$ satisfies the \emph{balancing condition} at $\tau$ if
\[
\sum_{\sigma\supset \tau} \text{w}_{\mathscr{C}}(\sigma) u_{\sigma/\tau}=0  
\]
in the complex vector space $N(\tau)_\mathbb{C}$, where the sum is over all $p$-dimensional cells $\sigma$ in $\mathscr{C}$ containing $\tau$ as a face.
A weighted complex is \emph{balanced} if it satisfies the balancing condition at each of its codimension $1$ cells.
\end{definition}

A \emph{tropical variety} is a positive and balanced weighted complex with finitely many cells, and a \emph{tropical current} is the current associated to a tropical variety.
Our first main result is the following criterion for the closedness of $\mathscr{T}_\mathscr{C}$, cf. \cite[Theorem 3.1.8]{Babaee}.

\begin{theorem}\label{Main1}
A weighted complex $\mathscr{C}$ is balanced if and only if $\mathscr{T}_{\mathscr{C}}$ is closed.
\end{theorem}

Theorem \ref{Main1} follows from an explicit formula for the boundary of $\mathscr{T}_\mathscr{C}$ in Theorem \ref{Main1General}:
\[
d \mathscr{T}_\mathscr{C}=-\sum_\tau  \mathscr{A}_\tau\Big(\sum_{ \tau \subset \sigma} \text{w}_\mathscr{C}(\sigma)u_{\sigma/\tau}\Big).
\]
Here the first sum is over all codimension $1$ cells $\tau$ in $\mathscr{C}$, the second sum is over all $p$-dimensional cells $\sigma$ in $\mathscr{C}$ containing $\tau$, and $\mathscr{A}_\tau$ is an injective linear map constructed in Section \ref{SectionAveraging} using the averaging operator of the compact Lie group $S_{N(\tau)}$.

\subsection{}

Some properties of the current $\mathscr{T}_\mathscr{C}$ can be read off from the polyhedral geometry of $|\mathscr{C}|$.
We show that this is the case for the property of $\mathscr{T}_\mathscr{C}$ being strongly extremal.

\begin{definition}
A closed current $\mathscr{T}$ with measure coefficients is \emph{strongly extremal} if for any closed current $\mathscr{T}'$ with measure coefficients which has the same dimension and  support as $\mathscr{T}$ there is  a complex number $c$ such that $\mathscr{T}'=c \cdot \mathscr{T}$.
\end{definition}

If $\mathscr{T}$ is strongly positive and strongly extremal, then  $\mathscr{T}$ generates an extremal ray in the cone of strongly positive closed currents:
If $\mathscr{T}=\mathscr{T}_1+\mathscr{T}_2$ is any decomposition of $\mathscr{T}$ into strongly positive closed currents, then 
both $\mathscr{T}_1$ and $\mathscr{T}_2$ are nonnegative multiples of $\mathscr{T}$.
Indeed, we have
\[
|\mathscr{T}|=|\mathscr{T}+\mathscr{T}_1|=|\mathscr{T}+\mathscr{T}_2|,
\]
and hence there are constants $c_1$ and $c_2$ satisfying
\[
\mathscr{T}+\mathscr{T}_1=c_1 \cdot \mathscr{T}, \qquad \mathscr{T}+\mathscr{T}_2=c_2 \cdot \mathscr{T}, \qquad c_1,c_2 \ge 1.
\]

\begin{definition}
A balanced weighted complex $\mathscr{C}$ is \emph{strongly extremal} if for any balanced weighted complex $\mathscr{C}'$ which has the same dimension and  support as $\mathscr{C}$ there is a complex number $c$ such that $\mathscr{C}' \sim c \cdot \mathscr{C}$. 
\end{definition}

A weighted complex in $\mathbb{R}^n$ is said to be \emph{non-degenerate} if its support is contained in no proper affine subspace of $\mathbb{R}^n$.
Our second main result provides a  new class of strongly extremal closed currents on $(\mathbb{C}^*)^n$.

\begin{theorem}\label{Main2}
A non-degenerate balanced weighted complex $\mathscr{C}$ is strongly extremal if and only if $\mathscr{T}_\mathscr{C}$ is strongly extremal.
\end{theorem}

This follows from Fourier analysis for tropical currents developed in the next section. 
A $0$-dimensional weighted complex in $\mathbb{R}^1$ shows that the assumption of non-degeneracy is necessary in Theorem \ref{Main2}, as the corresponding measure $\mu$ on $\text{Log}^{-1}(\{\text{pt}\})$ can be chosen arbitrarily. 

\begin{remark}
We note that there is an abundance of strongly extremal tropical varieties. For example, the Bergman fan of any simple matroid is a strongly extremal tropical variety; see \cite[Chapter III]{Huh} for the Bergman fan and the extremality.
Let $\mathscr{T}_M$ be the tropical current associated to the Bergman fan of a simple matroid $M$ on the ground set $\{0,1,\ldots,n\}$. If $M$ is representable over $\mathbb{C}$, then by \cite[Theorem 5.27]{Babaee} there are closed subvarieties $Z_i \subseteq (\mathbb{C}^*)^n$ of the ambient torus and positive real numbers $\lambda_i$ such that
\[
\mathscr{T}_M=\lim_{i \to \infty} \lambda_i [Z_i].
\]
It would be interesting to know whether $\mathscr{T}_M$ can be approximated as above when $M$ is not representable over $\mathbb{C}$. See \cite[Section 4.3]{Huh} for a related discussion.

\end{remark}


We end this section with a useful sufficient condition for the strong extremality of $\mathscr{C}$.

\begin{definition}
Let $\mathscr{C}$ be a $p$-dimensional weighted complex in $\mathbb{R}^n$.
\begin{enumerate}[(1)]
\item $\mathscr{C}$ is \emph{locally extremal} if, for every codimension $1$ cell $\tau$ in $\mathscr{C}$, every proper subset of
\[
\big\{u_{\sigma/\tau} \mid \text{$\sigma$ is a $p$-dimensional cell in $\mathscr{C}$ containing $\tau$ with nonzero $\text{w}_\mathscr{C}(\sigma)$}\big\}
\]
is linearly independent in the normal space $N(\tau)_\mathbb{R}$.
\item $\mathscr{C}$ is \emph{connected in codimension $1$} if, for every pair of $p$-dimensional cells $\sigma', \sigma''$ in $\mathscr{C}$ with nonzero weights, there are codimension $1$ cells $\tau_1,\ldots,\tau_l$ and  $p$-dimensional cells $\sigma_0,\sigma_1,\ldots,\sigma_l$ in $\mathscr{C}$ with nonzero weights such that
\[
\sigma'=\sigma_0 \supset \tau_1 \subset \sigma_1 \supset \tau_2 \subset \sigma_2 \supset \cdots \supset \tau_l \subset \sigma_l= \sigma''.
\]
\end{enumerate}
\end{definition}

The following sufficient condition for the strong extremality of $\mathscr{C}$ was used as a definition of strong extremality of $\mathscr{C}$ in \cite{Babaee}.

\begin{proposition}\label{LocalGlobalExtremality}
If a balanced weighted complex $\mathscr{C}$ is locally extremal and connected in codimension $1$, then it is strongly extremal.
\end{proposition}

\begin{proof}
Let $\mathscr{C}'$ be a $p$-dimensional balanced weighted complex with $|\mathscr{C}|=|\mathscr{C}'|$. 
We show that there is a complex number $c$ such that $\mathscr{C}' \sim c \cdot \mathscr{C}$.
Note that any refinement of $\mathscr{C}$ is balanced, locally extremal, and connected in codimension $1$. By replacing $\mathscr{C}$ and $\mathscr{C}'$ with their refinements, we may assume that 
the set of $p$-dimensional cells in $\mathscr{C}$ with nonzero weights is the set of $p$-dimensional cells in $\mathscr{C}'$ with nonzero weights.

We may suppose that $\mathscr{C}$ is not equivalent to $0$. 
Choose a $p$-dimensional cell $\sigma'$ in $\mathscr{C}$ with nonzero weight, and let $c$ be the complex number satisfying
\[
\text{w}_{\mathscr{C}'}(\sigma')=c \cdot \text{w}_\mathscr{C}(\sigma').
\]
We show that, for any other $p$-dimensional cell $\sigma''$ in $\mathscr{C}$ with nonzero weight,
\[
\text{w}_{\mathscr{C}'}(\sigma'')=c \cdot \text{w}_\mathscr{C}(\sigma'').
\]
Since $\mathscr{C}$ is connected in codimension $1$, there are codimension $1$ cells $\tau_1,\ldots,\tau_l$ and  $p$-dimensional cells $\sigma_0,\sigma_1,\ldots,\sigma_l$ in $\mathscr{C}$ with nonzero weights such that
\[
\sigma'=\sigma_0 \supset \tau_1 \subset \sigma_1 \supset \tau_2 \subset \sigma_2 \supset \cdots \supset \tau_l \subset \sigma_l=\sigma''.
\]
By induction on the minimal distance $l$ between $\sigma'$ and $\sigma''$ in $\mathscr{C}$, we are reduced to the case when $l=1$, that is, when $\sigma'$ and $\sigma''$ have a common codimension $1$ face $\tau$. The balancing conditions for $\mathscr{C}$ and $\mathscr{C}'$ at $\tau$ give
\[
\sum_{\sigma \supset \tau} \Big(\text{w}_{\mathscr{C}'}(\sigma)-c\cdot  \text{w}_\mathscr{C}(\sigma)\Big) u_{\sigma/\tau}=0,
\]
where the sum is over all $p$-dimensional cells $\sigma$ in $\mathscr{C}$ with nonzero weight that contain $\tau$. Since $\mathscr{C}$ is locally extremal, every proper subset of the vectors $u_{\sigma/\tau}$ is linearly independent, and hence
\[
\text{w}_{\mathscr{C}'}(\sigma')-c \cdot \text{w}_\mathscr{C}(\sigma')=0 \quad \text{implies} \quad \text{w}_{\mathscr{C}'}(\sigma'')-c \cdot \text{w}_\mathscr{C}(\sigma'')=0.
\]

\end{proof}

\section{Fourier analysis for tropical currents}\label{SectionFourier}

We develop necessary Fourier analysis on tori for proofs of Theorems \ref{Main1} and \ref{Main2}.

\subsection{}

Let $N$ be a finitely generated free abelian group, and let $M$ be the dual group $\text{Hom}_{\mathbb{Z}}(N,\mathbb{Z})$. 
The \emph{one-parameter subgroup} corresponding to $u \in N$ is the homomorphism
\[
\lambda^u:S^1 \longrightarrow S_N,\qquad z \longmapsto z \otimes u.
\]
The \emph{character} corresponding to $m \in M$ is the homomorphism
\[
\chi^m:S_N \longrightarrow S^1, \qquad z\otimes u \longmapsto z^{\langle u,m\rangle},
\]
where $\langle u,m\rangle$ denotes the dual pairing between elements of $N$ and $M$.

We orient the unit circle $S^1$ as the outer boundary of the complex manifold $\mathbb{D}^*$, the punctured unit disc in $\mathbb{C}^*$.
This makes each one-parameter subgroup of $S_N$ a $1$-dimensional current on $S_N$: 
The pairing between $\lambda^u$ and a smooth $1$-form $w$ is given by
\[
\langle \lambda^u,w\rangle:= \int_{S^1} (\lambda^u)^*w.
\]
We write $d\theta$ for the invariant $1$-form on $S^1$ with $\int_{S^1} d\theta =1$ corresponding to the chosen orientation. For $m \in M$, we define a smooth $1$-form $w(m)$ on $S_N$ by
\[
w(m):= (\chi^m)^*d\theta.
\]
Then we have
\[
\big\langle \lambda^u,w(m)\big\rangle=\int_{S^1} (\lambda^u)^*(\chi^m)^* \ d\theta=\langle u,m\rangle.
\]
Taking linear combinations of $1$-dimensional currents and smooth $1$-forms, the above gives the dual pairing between $N_\mathbb{C}$
and the dual Lie algebra of $S_N$. In particular, for $u_1,u_2 \in N$ and any invariant $1$-form $w$, we have
\[
\langle \lambda^{u_1+u_2},w\rangle=\langle \lambda^{u_1},w\rangle+
\langle \lambda^{u_2},w\rangle.
\]
Note however that, in general, $\lambda^{u_1+u_2} \neq \lambda^{u_1}+\lambda^{u_2}$ as $1$-dimensional currents on $S_N$.

We write $x^*(w)$ for the pullback of a smooth $1$-form $w$ along the multiplication map 
\[
\xymatrix{S_N \ar[r]^x& S_N}, \qquad x \in S_N.
\]

\begin{definition}
Let  $u \in N$, $m \in M$, and $\nu$ be a complex Borel measure on $S_{N}$. \begin{enumerate}[(1)]
\item The \emph{$\nu$-average} of a smooth $1$-form $w$ on $S_N$ is the smooth $1$-form
\[
\mathscr{A}(w,\nu):= \int_{x \in S_N} x^* (w) \ d\nu(x).
\]
\item The \emph{$\nu$-average} of $\lambda^u$ is the $1$-dimensional current $\mathscr{A}(\lambda^u,\nu)$ on $S_N$ defined by
\[
\big\langle\mathscr{A}(\lambda^u,\nu),w\big\rangle:= \int_{S^1} (\lambda^u)^* \mathscr{A}(w,\nu).
\]
\item The \emph{$m$-th Fourier coefficient}  of $\nu$ is the complex number
\[
\hat \nu(m):= \int_{x \in S_N} \chi^m \ d\nu(x)
\]
\end{enumerate}
When $\nu$ is the normalized Haar measure on $S_N$, we omit $\nu$ from the notation and write
\[
\mathscr{A}(w):=\mathscr{A}(w,\nu), \qquad \mathscr{A}(\lambda^u):= \mathscr{A}(\lambda^u,\nu).
\]
\end{definition}

We record here basic properties of the above objects.
Define
\[
\delta_k:=\begin{cases} 1& \text{if $k = 0$,}\\0&\text{if $k\neq 0$.}\end{cases}
\]

\begin{proposition}\label{PropositionFourier}
Let $u$ be an element of $N$, and let $m$ be an element of $M$.
\begin{enumerate}[(1)]
\item If $w$ is an invariant $1$-form on $S_N$, then
\[
\mathscr{A}(\chi^m w,\nu)=\hat \nu(m) \cdot \chi^m  w.
\]
\item If $w$ is an invariant $1$-form on $S_N$, then
\[
\langle \lambda^u,\chi^m w\rangle=\delta_{\langle u,m\rangle} \cdot \langle \lambda^u,w\rangle.
\]
\item If $w$ is an invariant $1$-form on $S_N$, then
\[
\big\langle \mathscr{A}(\lambda^u,\nu),\chi^m w\big\rangle=\delta_{\langle u,m\rangle}  \cdot \hat \nu(m) \cdot \langle \lambda^u, w \rangle.
\]
\end{enumerate}
\end{proposition}

\begin{proof}
Since $w$ is invariant and $\chi^m$ is a homomorphism, for each $x \in S_N$, we have
\[
x^*(\chi^m w)=x^*(\chi^m) \cdot x^*(w) = \chi^m(x) \ \chi^m \cdot w.
\]
Therefore,
\[
\mathscr{A}(\chi^m w,\nu)=\int_{x \in S_N} x^* (\chi^m w) \ d\nu(x)=\hat \nu(m) \cdot \chi^m w.
\]
This proves the first item. 

The second item follows from the computation
\[
\langle \lambda^u,\chi^m w\rangle=  \int_{S^1} (\lambda^u)^* (\chi^m w)=  \int_{S^1} z^{\langle u,m\rangle} (\lambda^u)^* w.
\]
The last integral is zero unless $\langle u,m\rangle$ is zero, because $(\lambda^u)^* w$ is an invariant $1$-form.

The third item is a combination of the first two:
\begin{eqnarray*}
\big\langle \mathscr{A}(\lambda^u,\nu),\chi^m w\big\rangle
=\big\langle \lambda^u,\mathscr{A}(\chi^m w,\nu)\big\rangle
=\delta_{\langle u,m\rangle} \cdot \hat \nu(m) \cdot \big\langle \lambda^u,\chi^m w\big\rangle.
\end{eqnarray*}
\end{proof}

Consider the split exact sequence associated to a primitive element $u$ of $N$:
\[
\xymatrix{0 \ar[r] & S^1 \ar[r]^{\lambda^u} & S_N \ar[r]^{q_u\quad\ } & \text{coker}(\lambda^u) \ar[r] & 0.}
\]
Let $\mu$ be a complex Borel measure on the cokernel of $\lambda^u$, $\mu_1$ the normalized Haar measure on $S^1$, and $\nu$ the pullback of the product measure $\mu \times \mu_1$ under a splitting isomorphism
\[
S_N \simeq \text{coker}(\lambda^u) \times S^1.
\]
Each fiber of the submersion $q_u$ is a translations of the image of $\lambda^u$ in $S_N$,  equipped with the orientation induced from that of $S^1$.

\begin{proposition}\label{GeometricI}
If $u$ is a primitive element of $N$, then
\[
\mathscr{A}(\lambda^u,\nu)=\int_{x \in \text{coker}(\lambda^u)} \big[q_u^{-1}(x)\big] d\mu(x).
\]
In particular, if $\mu$ is the normalized Haar measure on the cokernel of $\lambda^u$, then
\[
\mathscr{A}(\lambda^u)=\int_{x \in \text{coker}(\lambda^u)} \big[q_u^{-1}(x)\big] d\mu(x).
\]
\end{proposition}

\begin{proof}
By Fubini's theorem, for any smooth $1$-form $w$ on $S_N\simeq \text{coker}(\lambda^u) \times S^1$,
\begin{eqnarray*}
\big\langle \mathscr{A}(\lambda^u,\nu),w\big\rangle&=& \int_x \int_y \Bigg(  \int_{S^1} (\lambda^u)^*  x^* y^*w  \Bigg) d\mu(x)\ d\mu_1(y)\\
&=& \int_x \Bigg(  \int_{S^1} (\lambda^u)^*  x^* w  \Bigg) d\mu(x) \cdot \int_{y} d\mu_1(y)\\
&=& \int_x \Bigg(   \int_{q_u^{-1}(x)}  w \Bigg) d\mu(x). 
\end{eqnarray*}
This shows the equality between $1$-dimensional currents
\[
\mathscr{A}(\lambda^u,\nu)=\int_{x \in \text{coker}(\lambda^u)} \big[q_u^{-1}(x)\big] d\mu(x).
\]
If $\mu$ is the normalized Haar measure on $\text{coker}(\lambda^u)$, then $\nu$ is the normalized Haar measure on $S_N$, and the second statement follows.
\end{proof}

In other words, when $u$ is a primitive element of $N$, $\mathscr{A}(\lambda^u,\nu)$ is the pullback of the $0$-dimensional current $d\mu$ along the oriented submersion $q_u$.
In general,
\[
\mathscr{A}(\lambda^u,\nu)=m_u  \int_{x \in \text{coker}(\lambda^u)} \big[q_u^{-1}(x)\big] d\mu(x),
\]
where $m_u$ is the nonnegative integer satisfying 
$u=m_u u'$ with $u'$ primitive.

\subsection{}\label{SectionAveraging}
Let us recall a definition of pullback of a current (see \cite[Chapter 1, 2.15]{DemaillyBook1} for details). Consider a submersion 
$\pi: M\rightarrow N$, of complex manifolds $M$ and $N$, with respective complex dimensions $m$ and $n.$ Let $\varphi$ be a differential form of degree $k$ on $X'$, with $L^1_{\text{loc}}$ coefficients, such that the restriction $\pi_{|_{\text{Supp}}\varphi}$ is proper. Then the form
 $$\pi_{*}\varphi:= \int_{z\in F^{-1}(y)}\varphi(z),$$
is in $ \mathcal{D}^{k-2(m-n)}(N).$ Therefore, for a current $\mathscr{T}\in \mathcal{D}_{k-2(n-m)}'(M)$ the \emph{pullback} of $\mathscr{T}$ by $\pi$, $\pi^*\mathscr{T}\in \mathcal{D}_{k}'(N),$ is obtained by
$$
\langle \pi^*\mathscr{T}, \varphi \rangle = \langle \mathscr{T}, \pi_*\varphi \rangle.
$$ 
Note that for an analytic cycle $Z$, $\pi^*[Z]= [\pi^{-1}Z],$ if $\pi$ is a diffeomorphism. 

Now let $\tau$ be a rational polyhedron in $\mathbb{R}^n$. 
Let $\pi_{\text{aff}(\tau)}$ is the submersion associated to $\text{aff}(\tau)$ and $\iota^{\text{aff}(\tau)}$ is the closed embedding in the diagram
\[
\xymatrix{
\text{Log}^{-1}\big(\text{aff}(\tau)\big) \ar[r]^{\qquad \iota^{\text{aff}(\tau)}} \ar[d]_{\pi_{\text{aff}(\tau)}} & (\mathbb{C}^*)^n\\
S_{N(\tau)}.&
}
\]
For $u \in N(\tau)$ and a complex Borel measure $\nu$ on $S_{N(\tau)}$, we define a current on $(\mathbb{C}^*)^n$ by
\[
\mathscr{A}_\tau(u,\nu):=\mathbf{1}_{\text{Log}^{-1}(\tau)}\ \iota^{\text{aff}(\tau)}_*\Big( \pi_{\text{aff}(\tau)}^* \  \mathscr{A}(\lambda^u,\nu) \Big).
\]
In other words, $\mathscr{A}_\tau(u,\nu)$ is the trivial extension of the pullback of the $\nu$-average of $\lambda^u$ along the oriented submersion $\pi_\tau$.
When $\nu$ is the normalized Haar measure on $S_{N(\tau)}$, we write
\[
\mathscr{A}_\tau(u):=\mathscr{A}_\tau(u,\nu).
\]
For any nonzero $u$, the support of $\mathscr{A}_\tau(u,\nu)$ satisfies
\[
\big|\mathscr{A}_\tau(u,\nu)\big|\subseteq \big|\mathscr{A}_\tau(u)\big|=\text{Log}^{-1}(\tau) \subseteq (\mathbb{C}^*)^n.
\]

\begin{proposition}\label{Linearity}
For any $u_1,u_2 \in N(\tau)$,
\[
\mathscr{A}_\tau(u_1+u_2)=\mathscr{A}_\tau(u_1)+\mathscr{A}_\tau(u_2).
\]
\end{proposition}

\begin{proof}
Since  $\pi_\tau^*$ is linear, it is enough to check that $\mathscr{A}$ is linear.
Fourier coefficients of the normalized Haar measure $\nu$ on $S_{N(\tau)}$ are
\[
\hat \nu(m)=\begin{cases} 1& \text{if $m=0$,}\\ 0& \text{if $m \neq 0$.}\end{cases}
\]
Therefore, by Proposition \ref{PropositionFourier},  for any character $\chi^m$ and invariant $1$-form $w$ on $S_{N(\tau)}$, 
\[
\big\langle \mathscr{A}(\lambda^{u_1+u_2}),\chi^mw \big\rangle = \begin{cases} \langle \lambda^{u_1},w\rangle+\langle \lambda^{u_2},w\rangle& \text{if $m=0$,}\\ 0& \text{if $m \neq 0$,}\end{cases}
\]
and
\[
\big\langle \mathscr{A}(\lambda^{u_1}),\chi^mw \big\rangle+\big\langle \mathscr{A}(\lambda^{u_2}),\chi^mw \big\rangle = \begin{cases} \langle \lambda^{u_1},w\rangle+\langle \lambda^{u_2},w\rangle& \text{if $m=0$,}\\ 0& \text{if $m \neq 0$.}\end{cases}
\]
The Stone-Weierstrass theorem shows that any smooth $1$-form on $S_{N(\tau)}$ can be uniformly approximated by linear combinations of $1$-forms of the form $\chi^m w$ with $w$ invariant, and hence the above implies 
\[
\mathscr{A}(\lambda^{u_1+u_2})=\mathscr{A}(\lambda^{u_1})+\mathscr{A}(\lambda^{u_2}).
\]
\end{proof}

We note that the linear operators $\mathscr{A}_\tau$ and $\mathscr{A}$ are injective: By Proposition \ref{PropositionFourier}, for any element $m$ in the dual group  $M(\tau):=N(\tau)^\vee$,
\[
\big\langle \mathscr{A}(\lambda^u),w(m) \big\rangle=\langle \lambda^u,w(m)\rangle=\langle u,m\rangle.
\]
It follows that  $\mathscr{A}_\tau(u)=0$ if and only if $\mathscr{A}(\lambda^u)=0$ if and only if  $u=0$.

\subsection{}

Let $\tau$ be a codimension $1$ face of a $p$-dimensional rational polyhedron $\sigma$ in $\mathbb{R}^n$. 
Corresponding to each point $x \in S_{N(\sigma)}$, there is a commutative diagram of maps between smooth manifolds
\[
\xymatrix{
&&l_{\sigma,x}^{-1}(\sigma^\circ)\ar[d]&\\
&l_{\sigma,x}^{-1}(\tau^\circ) \ar[d] \ar[r]&\pi_{\text{aff}(\sigma)}^{-1}(x) \ar[d] \ar[dr]^{ l_{\sigma,x}}&&\\
&\text{Log}^{-1}(\tau^\circ) \ar[d]^{\pi_\tau}  \ar[r]& \text{Log}^{-1}\big(\text{aff}(\sigma)\big) \ar[d]^{\pi_{\text{aff}(\sigma)}} \ar[r]_{\qquad l_{\text{aff}(\sigma)}}& \text{aff}(\sigma)\\
S^1 \ar[r]^{\lambda^{u_{\sigma/\tau}}}&S_{N(\tau)} \ar[r]^{q_{\sigma/\tau}}& S_{N(\sigma)}. &
}
\]
The maps $\pi_\tau$, $\pi_{\text{aff}(\sigma)}$ are submersions with oriented fibers, the maps $l_{\sigma,x}$, $l_{\text{aff}(\sigma)}$ are restrictions of the logarithm map, and
all unlabeled maps are inclusions between subsets of $(\mathbb{C}^*)^n$. 
The dimensions of the above manifolds are depicted in the following diagram:
\[
\tiny
\xymatrix{
&&2p \ar[d]&\\
&2p-1 \ar[d] \ar[r]&2p \ar[d] \ar[dr]&&\\
&n+p-1 \ar[d]  \ar[r]& n+p \ar[d] \ar[r]& p\\
1 \ar[r] &n-p+1 \ar[r]& n-p. &
}
\]
The bottom row is a split exact sequence of Lie groups, and there is a canonical isomorphism
\[
S_{N(\sigma)} \simeq \text{coker}(\lambda^{u_{\sigma/\tau}}).
\]
Each fiber of the submersion $q_{\sigma/\tau}$ has the orientation induced from that of $S^1$.

\begin{lemma}\label{OrientationLemma}
We have the following equality between currents on $\text{Log}^{-1}(\tau^\circ)$:
\[
\big[l_{\sigma,x}^{-1}(\tau^\circ)\big] = \pi_\tau^{*} \big[q_{\sigma/\tau}^{-1}(x)\big].
\]
\end{lemma}

\begin{proof}
By construction, the top square in the diagram is cartesian:
\[
l_{\sigma,x}^{-1}(\tau^\circ) =\text{Log}^{-1}(\tau^\circ) \cap \pi_{\text{aff}(\sigma)}^{-1}(x).
\]
This equality, together with the commutativity of the two squares, shows that
\[
l_{\sigma,x}^{-1}(\tau^\circ) = \pi_\tau^{-1} \big(q_{\sigma/\tau}^{-1}(x)\big).
\]
The left-hand side is oriented as a boundary of the complex manifold $l_{\sigma,x}^{-1}(\sigma^\circ)$, and the circle
$q_{\sigma/\tau}^{-1}(x)$ is oriented as a translate of the one-parameter subgroup $\lambda^{u_{\sigma,\tau}}$. 
The canonical orientation on fibers of $\pi_\tau$ gives the orientation on the right-hand side.

We show that the two orientations agree. We do this explicitly after three reduction steps:
\begin{enumerate}[(1)]
\item It is enough to show this locally around any one point in $l_{\sigma,x}^{-1}(\tau^\circ)$. Therefore, we may assume that $\tau=\text{aff}(\tau)$.
\item By translation, 
we may assume that the chosen point is the identity element of the ambient torus.
\item By monomial change of coordinates,  we may assume that
\[
\tau=\text{span}(e_2,\ldots,e_p), \qquad \sigma=\text{cone}(e_1)+\tau.
\]
\end{enumerate}
Here $e_1,\ldots,e_n$ is the standard basis of $\mathbb{Z}^n$.
Recall that the punctured unit disc $\mathbb{D}^*$ maps to the positive real line $\mathbb{R}_{>0}$ under the logarithm map.
Under the above assumptions, the diagram reads
\[
\tiny
\xymatrix{
&&\mathbb{D}^* \times (\mathbb{C}^*)^{p-1} \times \{1\} \ar[d]&\\
&S^1 \times (\mathbb{C}^*)^{p-1} \times \{1\}\ar[d] \ar[r]& \mathbb{C}^* \times (\mathbb{C}^*)^{p-1} \times \{1\} \ar[d] \ar[dr]&&\\
&S^1 \times (\mathbb{C}^*)^{p-1} \times (S^1)^{n-p} \ar[d]  \ar[r]&  \mathbb{C}^* \times (\mathbb{C}^*)^{p-1} \times (S^1)^{n-p} \ar[d] \ar[r]& \mathbb{R} \times \mathbb{R}^{p-1} \times \{0\}\\
S^1 \times \{1\} \times \{1\} \ar[r] & S^1 \times \{1\} \times (S^1)^{n-p} \ar[r]& \{1\} \times \{1\}\times (S^1)^{n-p}. &
}
\]
From this diagram we see that the orientation on $l_{\sigma,x}^{-1}(\tau^\circ)$ as a boundary of $l_{\sigma,x}^{-1}(\sigma^\circ)$ agrees with the product of the orientation on $S^1$ and the canonical orientation on fibers of $\pi_\tau$.
\end{proof}

It follows that there is an equality between the trivial extensions to $(\mathbb{C}^*)^n$
\[
\big[l_{\sigma,x}^{-1}(\tau)\big]=\big[\overline{\pi_\tau^{-1} (q_{\sigma/\tau}^{-1}(x))}\big].
\]

\subsection{}

Let $\sigma$ be a $p$-dimensional rational polyhedron in $\mathbb{R}^n$, and let $\mu_\sigma$ be a complex Borel measure on $S_{N(\sigma)}$. For each codimension $1$ face $\tau$ of $\sigma$, consider the split exact sequence
\[
\xymatrix{
0 \ar[r] & S^1 \ar[r]^{\lambda^{u_{\sigma/\tau}}} & S_{N(\tau)} \ar[r]^{q_{\sigma/\tau}} & S_{N(\sigma)} \ar[r] & 0.
}
\]
Let $\nu_{\sigma/\tau}$ be the pullback of the product measure $\mu_\sigma \times \mu_1$ under a splitting isomorphism
\[
S_{N(\tau)} \simeq S_{N(\sigma)}\times S^1.
\]

\begin{proposition}\label{BoundaryComputation}
We have
\[
d\mathscr{T}_\sigma(\mu_\sigma)=-\sum_{\tau \subset \sigma}\mathscr{A}_\tau(u_{\sigma/\tau},\nu_{\sigma/\tau}),
\]
where the sum is over all codimension $1$ faces $\tau$ of $\sigma$.
In particular,
\[
d\mathscr{T}_\sigma=-\sum_{\tau \subset \sigma}\mathscr{A}_\tau(u_{\sigma/\tau})
\]
\end{proposition}

\begin{proof}
We start from the geometric representation of the boundary in Proposition \ref{BoundaryPolyhedron}. We have
\[
d\mathscr{T}_\sigma(\mu_\sigma)= -\sum_{\tau \subset \sigma} \Bigg(\int_{x \in S_{N(\sigma)}}    \big[l_{\sigma,x}^{-1}(\tau)\big] d\mu_\sigma(x) \Bigg).
\]
Lemma \ref{OrientationLemma} and Proposition \ref{GeometricI} together give
\[
d\mathscr{T}_\sigma(\mu_\sigma)=-\sum_{\tau \subset \sigma} \Bigg( \int_{x \in S_{N(\sigma)}} \big[\overline{\pi_\tau^{-1} (q_{\sigma/\tau}^{-1}(x))}\big] d\mu_{\sigma}(x) \Bigg)=-\sum_{\tau \subset \sigma}\mathscr{A}_\tau(u_{\sigma/\tau},\nu_{\sigma/\tau}).
\]
If $\mu_\sigma$ is the normalized Haar measure on $S_{N(\sigma)}$, then $\nu_{\sigma/\tau}$ is the normalized Haar measure on $S_{N(\tau)}$ for all $\tau \subset \sigma$, and the second statement follows.

\end{proof}

Let $\sigma$ and $\tau$ be as above, and consider the dual exact sequences
\[
\xymatrix{
0 \ar[r]& \mathbb{Z} \ar[r]^{u_{\sigma/\tau}\quad  } & N(\tau) \ar[r]& N(\sigma) \ar[r] & 0
}
\]
and
\[
\xymatrix{
0 \ar[r]& M(\sigma) \ar[r] & M(\tau) \ar[r]^{u_{\sigma/\tau}^\vee}& \mathbb{Z}^\vee \ar[r] & 0.
}
\]
The latter exact sequence shows that an element $m$ of $M(\tau)$ is in $M(\sigma)$ 
 if and only if 
 \[
 \langle u_{\sigma/\tau},m\rangle= 0.
 \]
When $m$ satisfies this condition,  the $m$-th Fourier coefficients of both $\nu_{\sigma/\tau}$ and $\mu_\sigma$ are defined.

\begin{proposition}\label{FourierCoefficient}
If an element $m$ of $M(\tau)$ is in $M(\sigma)$, then $\hat \nu_{\sigma/\tau}(m) = \hat \mu_\sigma(m)$.
\end{proposition}

\begin{proof}
Since $m \in M(\sigma)$, the character $\chi^m$ is constant along each fiber of $q_{\sigma/\tau}$. Therefore, by Fubini's theorem,
\[
\hat\nu_{\sigma/\tau}(m) =  \int_x  \chi^m \ d\mu(x) \cdot \int_{y} d\mu_1(y)= \hat \mu_{\sigma}(m).
\]
\end{proof}

The following formula for the boundary of $\mathscr{T}_\mathscr{C}$ directly implies Theorem \ref{Main1}.

\begin{theorem}\label{Main1General}
For any $p$-dimensional weighted complex $\mathscr{C}$ in $\mathbb{R}^n$,
\[
d \mathscr{T}_\mathscr{C}=-\sum_\tau  \mathscr{A}_\tau\Big(\sum_{ \tau \subset \sigma} \text{w}_\mathscr{C}(\sigma)u_{\sigma/\tau}\Big),
\]
where the second sum is over all $p$-dimensional cells $\sigma$ containing $\tau$.
\end{theorem}


\begin{proof}
By Proposition \ref{BoundaryComputation}, we have
\[
d\mathscr{T}_\mathscr{C}=-\sum_\sigma \sum_{\tau \subset \sigma} \text{w}_\mathscr{C}(\sigma) \mathscr{A}_\tau(u_{\sigma/\tau}).
\]
Changing the order of summation and applying Proposition \ref{Linearity} gives
\[
d\mathscr{T}_\mathscr{C}=-\sum_\tau  \mathscr{A}_\tau\Big(\sum_{ \tau \subset \sigma} \text{w}_\mathscr{C}(\sigma)u_{\sigma/\tau}\Big).
\]
\end{proof}

\subsection{}

Let $\mathscr{P}$ be a $p$-dimensional locally finite rational polyhedral complex in $\mathbb{R}^n$. 
We choose a complex Borel measure $\mu_\sigma$ on $S_{N(\sigma)}$ for each $p$-dimensional cell $\sigma$ of $\mathscr{P}$, and define a current
\[
\mathscr{T}:=\sum_\sigma \mathscr{T}_\sigma(\mu_\sigma),
\]
where the sum is over all $p$-dimensional cells $\sigma$ in $\mathscr{P}$. The support of $\mathscr{T}$ satisfies
\[
|\mathscr{T}| \subseteq \text{Log}^{-1}|\mathscr{P}|.
\]
In fact, any $(p,p)$-dimensional closed current with measure coefficients and support in $\text{Log}^{-1}|\mathscr{P}|$ is equal to $\mathscr{T}$ for some choices of complex Borel measures $\mu_\sigma$, see Lemma \ref{SupportLemma}.
For each $\sigma$ and its codimension $1$ face $\tau$, there are inclusion maps
\[
\xymatrix{
M(\sigma) \ar[r] & M(\tau) \ar[r] & (\mathbb{Z}^n)^\vee,
}
\]
dual to the quotient maps
\[
\xymatrix{
\mathbb{Z}^n \ar[r] & N(\tau) \ar[r] & N(\sigma).
}
\]
Let $m$ be an element of $(\mathbb{Z}^n)^\vee$. For each $p$-dimensional cell $\sigma$ in $\mathscr{P}$, we set
\[
\text{w}_\mathscr{T}(\sigma,m):=\begin{cases} \hat \mu_\sigma(m) & \text{if $m \in M(\sigma)$,} \\ \quad 0 & \text{if $m \notin M(\sigma)$.}\end{cases}
\]
This defines $p$-dimensional weighted complexes $\mathscr{C}_\mathscr{T}(m)$ in $\mathbb{R}^n$ satisfying
\[
|\mathscr{C}_\mathscr{T}(m)| \subseteq |\mathscr{P}|.
\]

\begin{theorem}\label{ClosedBalanced}
The current $\mathscr{T}$ is closed if and only if $\mathscr{C}_\mathscr{T}(m)$ is balanced for all $m \in (\mathbb{Z}^n)^\vee$.
\end{theorem}

When all the measures $\mu_\sigma$ are invariant, $\mathscr{C}_\mathscr{T}(m)$ is zero for all nonzero $m$, and Theorem \ref{ClosedBalanced} is equivalent to Theorem \ref{Main1}. 
The general case of Theorem \ref{ClosedBalanced} will be used in the proof of Theorem \ref{Main2}.

\begin{proof}
Let $\tau$ be a codimension $1$ cell in $\mathscr{P}$, and let $m$ be an element of $(\mathbb{Z}^n)^\vee$. If $m \notin M(\tau)$, then
for all $p$-dimensional cells $\sigma$ in $\mathscr{P}$ containing $\tau$,
\[
\text{w}_\mathscr{T}(\sigma,m)=0,
\]
and $\mathscr{C}_\mathscr{T}(m)$ trivially satisfies the balancing condition at $\tau$.
It remains to show that $\mathscr{T}$ is closed if and only if $\mathscr{C}_\mathscr{T}(m)$ satisfies the balancing condition at $\tau$ whenever $m \in M(\tau)$. 
By Proposition \ref{BoundaryComputation}, we have the expression
\[
d\mathscr{T}=-\sum_\tau \sum_{\tau \subset \sigma} \mathscr{A}_\tau(u_{\sigma/\tau},\nu_{\sigma/\tau}),
\]
where the second sum is over all $p$-dimensional cell $\sigma$ containing $\tau$.
Therefore, $\mathscr{T}$ is closed if and only if, for each codimension $1$ cell $\tau$ of $\mathscr{P}$,
\[
\sum_{\tau \subset \sigma} \mathscr{A}_\tau(u_{\sigma/\tau}, \nu_{\sigma/\tau})=0.
\]
This happens if and only if, for each codimension $1$ cell $\tau$ of $\mathscr{P}$,
\[
\sum_{\tau \subset \sigma} \pi_\tau^* \ \mathscr{A}(\lambda^{u_{\sigma/\tau}}, \nu_{\sigma/\tau})=0.
\]
Since each $\pi_\tau^*$ is an injective linear map, the remark following Proposition 3.4 implies that this condition is equivalent to
\[
\sum_{\tau \subset \sigma} \mathscr{A}(\lambda^{u_{\sigma/\tau}}, \nu_{\sigma/\tau})=0, \quad \text{for each $\tau$}.
\]
By the Stone-Weierstrass theorem, any smooth $1$-form on $S_{N(\tau)}$ can be uniformly approximated by linear combinations of $1$-forms of the form $\chi^m w$, where $\chi^m$ is a character and $w$ is an invariant $1$-form on $S_{N(\tau)}$, and hence the above condition holds if and only if
\[
\sum_{\tau \subset \sigma} \Big\langle \mathscr{A}(\lambda^{u_{\sigma/\tau}}, \nu_{\sigma/\tau}), \chi^m w \Big\rangle=0 \quad \text{for each $\tau$},
\]
for all characters $\chi^m$ and all invariant $1$-forms $w$ on $S_{N(\tau)}$.
Using Propositions \ref{PropositionFourier} and \ref{FourierCoefficient},
the equation reads
\[
\sum_{\tau \subset \sigma}\text{w}_\mathscr{T}(\sigma,m) \ \langle \lambda^{u_{\sigma/\tau}},w\rangle=0.
\]
Finally, the dual pairing between $N(\tau)_\mathbb{C}$ and $M(\tau)_\mathbb{C}$ shows that the condition holds if and only if
the balancing condition
\[
\sum_{\tau \subset \sigma}\text{w}_\mathscr{T}(\sigma,m)  \ u_{\sigma/\tau}=0 
\]
is satisfied for all $\tau$ and all elements $m \in M(\tau)$.
\end{proof}

\subsection{}\label{ExtremalProof}

Theorem \ref{Main1} can be used to prove one direction of Theorem \ref{Main2}. If $\mathscr{C}'$ is a balanced weighted complex which has the same dimension and support as $\mathscr{C}$, then $\mathscr{T}_{\mathscr{C}'}$ is a closed current with measure coefficients which has the same dimension and support as $\mathscr{T}_\mathscr{C}$. Therefore, if $\mathscr{T}_\mathscr{C}$ is strongly extremal, then there is a constant $c$ such that
\[
\mathscr{T}_{\mathscr{C}'}=c \cdot \mathscr{T}_\mathscr{C}= \mathscr{T}_{c \cdot \mathscr{C}}.
\]
This implies
\[
\mathscr{C}' \sim c \cdot \mathscr{C},
\]
and hence $\mathscr{C}$ is strongly extremal.
We prove the other direction after three lemmas.

\begin{lemma}\label{NonDegenerate}
A $p$-dimensional weighted complex $\mathscr{C}$ in $\mathbb{R}^n$ is non-degenerate
if and only if 
\[
\bigcap_\sigma M(\sigma)_\mathbb{R}=\{0\},
\]
where the intersection is over all $p$-dimensional cells in $\mathscr{C}$.
\end{lemma}

\begin{proof}
The non-degeneracy of $\mathscr{C}$ is equivalent to the exactness of
\[
\sum_\sigma H_\sigma \longrightarrow \mathbb{R}^n \longrightarrow 0,
\]
which is in turn equivalent to the exactness of
\[
0 \longrightarrow (\mathbb{R}^n)^\vee \longrightarrow \bigoplus_\sigma H_\sigma^\vee,
\]
where the sums are over all $p$-dimensional cells in $\mathscr{C}$.
The kernel of the latter map is the intersection of $M(\sigma)_\mathbb{R}$ in the statement of the lemma.
\end{proof}

\begin{lemma}\label{ProperSupport}
If the support of a balanced weighted complex $\mathscr{C}_1$ is properly contained in the support of a strongly extremal balanced weighted complex $\mathscr{C}_2$ of the same dimension, then $\mathscr{C}_1 \sim 0$.
\end{lemma}

\begin{proof}
The local finiteness of $\mathscr{C}_1$, $\mathscr{C}_2$ implies that there are only countably many cells in $\mathscr{C}_1$, $\mathscr{C}_2$. Therefore, there is a nonzero complex number $c_1$ such that
\[
\big|c_1\{ \mathscr{C}_1\}+\{\mathscr{C}_2\}\big|=\big|\mathscr{C}_2\big|.
\]
By the strong extremality of $\mathscr{C}_2$, there is a complex number $c_2$ with
\[
c_1\{ \mathscr{C}_1\}+\{\mathscr{C}_2\} = c_2 \{\mathscr{C}_2\}.
\]
Since the support of $\mathscr{C}_1$ is properly contained in the support of $\mathscr{C}_2$, the number $c_2$ should be $1$, and hence all the weights of $\mathscr{C}_1$ are zero.
\end{proof}

\begin{lemma}\label{SupportLemma}
Let $\mathscr{P}$ be a $p$-dimensional locally finite rational polyhedral complex in $\mathbb{R}^n$.
If the support of a $(p,p)$-dimensional current $\mathscr{T}$ with measure coefficients on $(\mathbb{C}^*)^n$ satisfies
\[
|\mathscr{T}| \subseteq \text{Log}^{-1}|\mathscr{P}|,
\]
then there are complex Borel measures $\mu_\sigma$ on $S_{N(\sigma)}$ such that
\[
\mathscr{T}=\sum_\sigma \mathscr{T}_\sigma(\mu_\sigma),
\]
where the sum is over all $p$-dimensional cells $\sigma$ in $\mathscr{P}$.
\end{lemma}

\begin{proof}
The second theorem on support \cite[Section III.2]{DemaillyBook1} implies that, for each $p$-dimensional cell $\sigma$ in $\mathscr{P}$, there is a complex Borel measure $\mu_\sigma$ on $S_{N(\sigma)}$ such that
\[
\mathscr{T}|_{\text{Log}^{-1}(\sigma^\circ)}=\pi_\sigma^*(d\mu_\sigma).
\]
The trivial extension of the right-hand side to $(\mathbb{C}^*)^n$ is by definition $\mathscr{T}_\sigma(\mu_\sigma)$, and hence
\[
\Big|\mathscr{T}-\sum_\sigma \mathscr{T}_\sigma(\mu_\sigma)\Big| \subseteq \bigcup_\tau \text{Log}^{-1}|\tau|,
\]
where the union is over all $(p-1)$-dimensional cells in $\mathscr{P}$. 
Note that each $\text{Log}^{-1}|\tau|$ is contained in the closed submanifold 
\[
\text{Log}^{-1}\big(\text{aff}(\tau)\big) \subseteq (\mathbb{C}^*)^n.
\]
Since this submanifold has Cauchy-Riemann dimension $p-1$, the first theorem on support \cite[Section III.2]{DemaillyBook1} implies that
\[
\mathscr{T}-\sum_\sigma \mathscr{T}_\sigma(\mu_\sigma)=0.
\]
\end{proof}

\begin{proof}[End of proof of Theorem \ref{Main2}]
Suppose $\mathscr{C}$ is non-degenerate and strongly extremal, and
let $\mathscr{T}$ be a closed current with measure coefficients which has the same dimension and support as $\mathscr{T}_\mathscr{C}$. Lemma \ref{SupportLemma} shows that there are complex Borel measures $\mu_\sigma$ on $S_{N(\sigma)}$ such that
\[
\mathscr{T}=\sum_\sigma \mathscr{T}_\sigma(\mu_\sigma),
\]
where the sum is over all $p$-dimensional cells $\sigma$ in $\mathscr{C}$.
For each $m \in (\mathbb{Z}^n)^\vee$, we construct the balanced weighted complexes $\mathscr{C}_\mathscr{T}(m)$ using Theorem \ref{ClosedBalanced}.
Since $\mathscr{C}$ is strongly extremal,  there are complex numbers $c(m)$ such that
\[
\mathscr{C}_\mathscr{T}(m) = c(m) \cdot \mathscr{C}, \qquad m \in (\mathbb{Z}^n)^\vee.
\]
Since $\mathscr{C}$ is non-degenerate, Lemma \ref{NonDegenerate} shows that the support of $\mathscr{C}_\mathscr{T}(m)$ is \emph{properly} contained in the support of $\mathscr{C}$
 for all nonzero $m  \in (\mathbb{Z}^n)^\vee$.
Therefore 
\[
\mathscr{C}_\mathscr{T}(m)=0, \qquad m \neq 0.
\]
In other words, the Fourier coefficient $\hat \mu_\sigma(m)$ is zero
for all $p$-dimensional cell $\sigma$ in $\mathscr{C}$ and all nonzero $m \in (\mathbb{Z}^n)^\vee$.
The measures $\mu_\sigma$ are determined by their Fourier coefficients, and hence each $\mu_\sigma$ is the invariant measure on $S_{N(\sigma)}$ with the normalization
\[
\int_{x \in S_{N(\sigma)}} d\mu_\sigma(x) = c(0).
\]
Therefore  $\mathscr{T}=c(0) \cdot \mathscr{T}_\mathscr{C}$,
and the current $\mathscr{T}_\mathscr{C}$ is strongly extremal.
\end{proof}

\section{Tropical currents on toric varieties}\label{SectionToric}

\subsection{}

Let $X$ be an $n$-dimensional smooth projective complex toric variety containing  $(\mathbb{C}^*)^n$, let $\Sigma$ be the fan of $X$, and let $p$ and $q$ be nonnegative integers satisfying $p+q=n$. Since $X$ is smooth $X\backslash (\mathbb{C}^*)^n$ is a simple normal crossing divisor, and the orbit closures are intersections of its components. 
A cohomology class in $X$ gives a homomorphism from the homology group of complementary dimension to $\mathbb{Z}$, defining the Kronecker duality homomorphism
\[
\mathscr{D}_X:H^{2q}(X,\mathbb{Z}) \longrightarrow \text{Hom}_\mathbb{Z}\big(H_{2q}(X,\mathbb{Z}),\mathbb{Z}\big), \qquad c \longmapsto \big(a \longmapsto \text{deg} (c \cap a)\big).
\]
The homomorphism $\mathscr{D}_X$ is, in fact, an isomorphism.
Since the homology group is generated by the classes of $q$-dimensional torus orbit closures, the duality identifies cohomology classes with certain $\mathbb{Z}$-valued functions on the set of $p$-dimensional cones in $\Sigma$, that is, with certain integral weights assigned to the $p$-dimensional cones in $\Sigma$.
The relation between homology classes of $q$-dimensional torus orbit closures of $X$ translates to the balancing condition on the integral weights on the $p$-dimensional cones in $\Sigma$ \cite[Theorem 2.1]{Fulton-Sturmfels}.

\begin{theorem}\label{KroneckerDuality}
The Kronecker duality gives isomorphisms between abelian groups
\[
H^{2q}(X,\mathbb{Z}) \simeq \text{Hom}\big(H_{2q}(X,\mathbb{Z}),\mathbb{Z}\big) \simeq \big\{\text{$p$-dimensional balanced integral weights on $\Sigma$}\big\},
\]
\end{theorem}

Therefore, by the Hodge decomposition theorem, cohomology group $H^{i}(X,\Omega_X^{j})$ vanishes when $i \neq j$, and that there is an induced isomorphism between complex vector spaces
\[
\mathscr{D}_{X,\mathbb{C}}:H^{q,q}(X) \longrightarrow  \big\{\text{$p$-dimensional balanced weights on $\Sigma$}\big\}.
\]
In other words, the Kronecker duality identifies elements of $H^{q,q}(X)$ with $p$-dimensional balanced weighted complexes in $\Sigma$.
Explicitly, for a smooth closed form $\varphi$ of degree $(q,q)$, 
\[
\mathscr{D}_{X,\mathbb{C}}:\{\varphi\} \longmapsto \Big( \gamma \longmapsto \int_{V(\gamma)} \varphi\Big),
\]
where $V(\gamma)$ is the $q$-dimensional torus orbit closure in $X$ corresponding to a $p$-dimensional cone $\gamma$ in $\Sigma$.

Let $w_0$ be the smooth positive $(1,1)$-form on $X$ corresponding to a fixed torus equivariant projective embedding 
\[
\phi: X \longrightarrow \mathbb{P}^N.
\]
The \emph{trace measure} of a $(p,p)$-dimensional positive current $\mathscr{T}$ on $X$ is the positive Borel measure
\[
\text{tr}(\mathscr{T})=\text{tr}(\mathscr{T},w_0):=\frac{1}{p!} \mathscr{T} \wedge w_0^p.
\]
The trace measure of a positive current on an open subset of $X$ is defined in the same way using the restriction of $w_0$.

\begin{proposition}\label{Finiteness} 
If $\mathscr{C}$ is a $p$-dimensional positive weighted complex in $\mathbb{R}^n$ with finitely many cells, then the trace measure of the positive current $\mathscr{T}_\mathscr{C}$ is finite.
\end{proposition}

\begin{proof}
Let $\sigma$ be a $p$-dimensional rational polyhedron in $\mathbb{R}^n$, and recall that each fiber $\pi_\sigma^{-1}(x)$ is an open subset of the $p$-dimensional closed subvariety $\pi_{\text{aff}(\sigma)}^{-1}(x) \subseteq (\mathbb{C}^*)^n$.
By Wirtinger's theorem \cite[Chapter 1]{Griffiths-Harris}, the normalized volume of $\pi_{\text{aff}(\sigma)}^{-1}(x)$ with respect to $w_0$ is the degree of the closure
\[
d_\sigma:=\text{deg}\Bigg( \overline{\pi_{\text{aff}(\sigma)}^{-1}(x)}^X \subseteq \mathbb{P}^N\Bigg).
\]
This integer $d_\sigma$ is independent of $x \in S_{N(\sigma)}$, because  the projective embedding $\phi$ is equivariant and fibers  of $\pi_{\text{aff}(\sigma)}$ are translates of each other under the action of $(S^1)^n$.
It follows that $\text{tr}(\mathscr{T}_\sigma) \le d_\sigma$, and hence
\[
\text{tr}(\mathscr{T}_\mathscr{C}) \le \sum_\sigma \text{w}_\mathscr{C}(\sigma) d_\sigma,
\]
where the sum is over all $p$-dimensional cells $\sigma$ in $\mathscr{C}$.
\end{proof}

Let $\mathscr{C}$ be a $p$-dimensional weighted complex  in $\mathbb{R}^n$ with finitely many cells. 
Proposition \ref{Finiteness} shows that $X$ is covered by coordinate charts $(\Omega,z)$ such that
\[
\mathscr{T}_\mathscr{C}|_{\Omega \cap (\mathbb{C}^*)^n}=\sum_{|I|=|J|=k} \mu_{IJ}\ dz_I \wedge d\bar z_J,
\]
where $\mu_{IJ}$ are complex Borel measures on $\Omega \cap (\mathbb{C}^*)^n$.
It follows that the current $\mathscr{T}_\mathscr{C}$ admits the \emph{trivial extension}, the current $\overline{\mathscr{T}}_\mathscr{C}$ on $X$ defined by
\[
\overline{\mathscr{T}}_\mathscr{C}|_{\Omega}=\sum_{|I|=|J|=k}  \nu_{IJ} \ dz_I \wedge d\bar z_J, 
\]
where $\nu_{IJ}$ are complex Borel measures on $\Omega$ given by $ \nu_{IJ}(-)=\mu_{IJ}\big(-\cap\hspace{0.5mm} (\mathbb{C}^*)^n\big)$.

\begin{lemma}\label{LinearCombination}
If $\mathscr{C}$ is a balanced weighted complex with finitely many cells, then 
there are complex numbers $c_1,\ldots,c_l$ and positive balanced weighted complexes $\mathscr{C}_1,\ldots,\mathscr{C}_l$ with finitely many cells such that
\[
\mathscr{T}_\mathscr{C}=\sum_{i=1}^l c_i \  \mathscr{T}_{\mathscr{C}_i}.
\]
\end{lemma}

\begin{proof}
Let $\mathscr{C}_p$ be the set of $p$-dimensional cells in $\mathscr{C}$, and consider the complex vector space
\[
W:=\Big\{\text{w}:\mathscr{C}_p \longrightarrow \mathbb{C} \mid \text{$\text{w}$ satisfies the balancing condition}\Big\}. 
\]
Since the balancing condition is defined over the real numbers, $W$ is spanned by elements of the form $\text{w}:\mathscr{C}_p \longrightarrow \mathbb{R}$.
Therefore it is enough to show the following statement: If $\mathscr{C}$ is a balanced weighted complex with real weights and finitely many cells, then $\mathscr{T}_\mathscr{C}$ can be written as a difference
\[
\mathscr{T}_\mathscr{C}=\mathscr{T}_\mathscr{A}-\mathscr{T}_\mathscr{B},
\]
where $\mathscr{A}$ and $\mathscr{B}$ are positive balanced weighted complexes with finitely many cells.

We construct the weighted complexes $\mathscr{A}$ and $\mathscr{B}$ from $\mathscr{C}$ as follows. Let $|\mathscr{A}|$ be the union
\[
 |\mathscr{A}|:=\bigcup_{\sigma \in \mathscr{C}_p} \ \text{aff}(\sigma),
\]
and note that there is a refinement of $\mathscr{C}$ that extends to a finite rational polyhedral subdivision of $|\mathscr{A}|$. Choose any such refinement $\mathscr{C}'$ of $\mathscr{C}$ and a subdivision $\mathscr{A}$ of $|\mathscr{A}|$. For each $p$-dimensional cell $\gamma$ in $\mathscr{A}$, we set
\[
\text{w}_\mathscr{A}(\gamma):=\max_{\sigma \in \mathscr{C}_p} \big|\text{w}_\mathscr{C}(\sigma)\big|, \qquad \text{w}_\mathscr{B}(\gamma):=\text{w}_\mathscr{A}(\gamma)-\text{w}_{\mathscr{C}'}(\gamma).
\]
This makes $\mathscr{A}$ and $\mathscr{B}$ positive weighted complexes satisfying
\[
\mathscr{T}_\mathscr{C}=\mathscr{T}_\mathscr{A}-\mathscr{T}_\mathscr{B}.
\]
It is easy to see that $\mathscr{A}$ is balanced, and $\mathscr{B}$ is balanced because  $\mathscr{A}$ and $\mathscr{C}$ are balanced.

\end{proof}

\begin{proposition}\label{ToricClosed}
If $\mathscr{C}$ is a balanced weighted complex with finitely many cells, then the trivial extension $\overline{\mathscr{T}}_\mathscr{C}$ is a closed current on $X$.
\end{proposition}

\begin{proof}
We use Lemma \ref{LinearCombination} to express $\mathscr{T}_\mathscr{C}$ as a linear combination
\[
\mathscr{T}_\mathscr{C}=\sum_{i=1}^l c_i\ \mathscr{T}_{\mathscr{C}_i},
\]
where $c_i$ are complex numbers and $\mathscr{C}_i$ are positive balanced weighted complexes with finitely many cells. By taking the trivial extension we have
\[
\overline{\mathscr{T}}_\mathscr{C}=\sum_{i=1}^l c_i\  \overline{\mathscr{T}}_{\mathscr{C}_i}.
\]
By Theorem \ref{Main1}, each $\mathscr{T}_{\mathscr{C}_i}$ is a positive closed current on the open subset $(\mathbb{C}^*)^n \subseteq X$. Since each $\mathscr{C}_i$ has finitely many cells, Proposition \ref{Finiteness} shows that Skoda's extension theorem \cite[Section III.2]{DemaillyBook1} applies to the positive closed current $\mathscr{T}_{\mathscr{C}_i}$. It follows that $d\overline{\mathscr{T}}_{\mathscr{C}_i}=0$, and hence
\[
d\overline{\mathscr{T}}_\mathscr{C}=\sum_{i=1}^l c_i\ d\overline{\mathscr{T}}_{\mathscr{C}_i}=0.
\]
\end{proof}

Any $(p,p)$-dimensional closed current $\mathscr{T}$ on $X$ defines a linear functional on $H^{p,p}(X)$:
\[
\mathscr{T} \longmapsto \Big(\psi \longmapsto \langle \mathscr{T},\psi\rangle\Big).
\]
Composing the above map with the Poincar\'e-Serre duality $H^{p,p}(X)^\vee \simeq H^{q,q}(X)$, we have
\[
 \mathscr{T} \longmapsto \{\mathscr{T}\} \in H^{q,q}(X).
\]
The element $\{\mathscr{T}\}$ is the \emph{cohomology class} of $\mathscr{T}$. 
In particular, a $p$-dimensional balanced weighted complex $\mathscr{C}$ with finitely many cells defines a cohomology class $\{\overline{\mathscr{T}}_\mathscr{C}\}$, which we  may view as a $p$-dimensional balanced weighted complex via Theorem \ref{KroneckerDuality}.
We compare these two balanced weighted complexes  in Theorem \ref{Main3}.

\subsection{}\label{SubsectionRecession}

Let $\mathscr{C}$ be  a $p$-dimensional balanced weighted complex in $\mathbb{R}^n$ with finitely many cells.
The \emph{recession cone} of a polyhedron $\sigma$ is the convex polyhedral cone
\[
\text{rec}(\sigma)=\{b \in \mathbb{R}^n \mid \sigma+b\subseteq \sigma\} \subseteq H_\sigma.
\]
If $\sigma$ is rational, then $\text{rec}(\sigma)$ is rational, and if $\sigma$ is a cone, then $\sigma=\text{rec}(\sigma)$. 

\begin{definition}\label{DefinitionCompatible}
We say that $\mathscr{C}$ is \emph{compatible} with $\Sigma$ if $\text{rec}(\sigma) \in \Sigma$ for all $\sigma \in \mathscr{C}$.
\end{definition}

There is a subdivision of $\mathscr{C}$ that is compatible with a subdivision of $\Sigma$, see \cite{Gil-Sombra}.

\begin{definition}
For each $p$-dimensional cone $\gamma$ in $\Sigma$, we define
\[
\text{w}_{\text{rec}(\mathscr{C})}(\gamma):= \sum_{\sigma} \text{w}_\mathscr{C}(\sigma),
\]
where the sum is over all $p$-dimensional cells $\sigma$ in $\mathscr{C}$ whose recession cone is $\gamma$. 
\end{definition}

This defines a $p$-dimensional weighted complex $\text{rec}(\mathscr{C},\Sigma)$, the \emph{recession} of $\mathscr{C}$ in $\Sigma$. When $\mathscr{C}$ is compatible with $\Sigma$, we write
\[
\text{rec}(\mathscr{C}):=\text{rec}(\mathscr{C},\Sigma).
\]
As suggested by the notation, the recession of $\mathscr{C}$ does not depend on $\Sigma$ when $\mathscr{C}$ is compatible with $\Sigma$. More precisely, if $\mathscr{C}_1 \sim \mathscr{C}_2$ and if $\mathscr{C}_i$ is compatible with a complete fan $\Sigma_i$ for $i=1,2$, then
\[
\text{rec}(\mathscr{C}_1,\Sigma_1) \sim \text{rec}(\mathscr{C}_2,\Sigma_2).
\]

\begin{theorem}\label{Main3}
If $\mathscr{C}$ is a $p$-dimensional tropical variety compatible with $\Sigma$, then
\[
\{\overline{\mathscr{T}}_\mathscr{C}\}=\text{rec}(\mathscr{C}) \in H^{q,q}(X).
\]
In particular, if all polyhedrons in $\mathscr{C}$ are cones in $\Sigma$, then
\[
\{\overline{\mathscr{T}}_\mathscr{C}\}=\mathscr{C} \in H^{q,q}(X).
\]
\end{theorem}
As a consequence, $\text{rec}(\mathscr{C})$ is a balanced complex, since it represents a cohomology class. 

The remainder of this section is devoted to the proof of Theorem \ref{Main3}.

\subsection{}\label{ToricLocalCoordinates}

Let $\sigma$ be a $p$-dimensional rational polyhedron in $\mathbb{R}^n$. If $\text{rec}(\sigma) \in \Sigma$,
we consider the corresponding torus invariant affine open subset
\[
U_{\text{rec}(\sigma)} := \text{Spec}\Big(\mathbb{C}[\text{rec}(\sigma)^\vee \cap \mathbb{Z}^n]\Big) \subseteq X.
\]
We write $p'$ for the dimension of the recession cone of $\sigma$, and $K_\sigma$ for the span of the recession cone of $\sigma$:
\[
p':=\dim\big(\text{rec}(\sigma)\big), \qquad K_\sigma:=\text{span}\big(\text{rec}(\sigma)\big) \simeq \mathbb{R}^{p'}.
\]
There are morphisms  between fans
\[
\big(\text{rec}(\sigma) \subseteq K_\sigma\big) \longrightarrow \big(\text{rec}(\sigma) \subseteq H_\sigma\big) \longrightarrow \big(\text{rec}(\sigma) \subseteq \mathbb{R}^n\big).
\]
Since $X$ is smooth, $\text{rec}(\sigma) \in \Sigma$ implies that $\text{rec}(\sigma)$ is unimodular, and the induced map between affine toric varieties fits into the commutative diagram
\[
\small
\xymatrix{
\overline{ T_{K_\sigma \cap \mathbb{Z}^n} }^{U_{\text{rec}(\sigma)}} \ar[r]  \ar[d]^{\varphi^1_\sigma}& \overline{ T_{H_\sigma \cap \mathbb{Z}^n} }^{U_{\text{rec}(\sigma)}} \ar[r] \ar[d]^{\varphi^2_\sigma}& U_{\text{rec}(\sigma)} \ar[d]^{\varphi^3_\sigma}\\
 \mathbb{C}^{p'} \ar[r]&  \mathbb{C}^{p'} \times (\mathbb{C}^*)^{p-p'}  \ar[r]  &  \mathbb{C}^{p'} \times (\mathbb{C}^*)^{n-p'},
}
\]
where $\varphi^1_\sigma,\varphi^2_\sigma,\varphi^3_\sigma$ are isomorphisms between toric varieties and the horizontal maps are equivariant closed embeddings.
We write $z_{\text{rec}(\sigma)}$ for the distinguished point of $U_{\text{rec}(\sigma)}$ corresponding to the semigroup homomorphism
\[
\text{rec}(\sigma)^\vee \cap \mathbb{Z}^n \longrightarrow \mathbb{C}, \qquad m \longmapsto \begin{cases} 1 & \text{if $m \in \sigma^\perp$,} \\ 0 & \text{if $m \notin \sigma^\perp$.}\end{cases}
\]
The isotropy subgroup of the distinguished point is $T_{K_\sigma \cap \mathbb{Z}^n} \subseteq (\mathbb{C}^*)^n$, and we may identify
$T_{N(\text{rec}(\sigma))}$ with the closed torus orbit of $U_{\text{rec}(\sigma)}$ by the map
\[
T_{N(\text{rec}(\sigma))} \longrightarrow U_{\text{rec}(\sigma)}, \qquad t \longmapsto t \cdot z_{\text{rec}(\sigma)}.
\]
Under the above commutative diagram, 
\[
\xymatrix{
z_{\text{rec}(\sigma)} \ar[r] \ar[d]& z_{\text{rec}(\sigma)} \ar[r] \ar[d]& z_{\text{rec}(\sigma)} \ar[d]\\
0_{ \mathbb{C}^{p'}} \ar[r]& 0_{ \mathbb{C}^{p'}} \times 1_{(\mathbb{C}^*)^{p-p'}} \ar[r]& 0_{ \mathbb{C}^{p'}} \times 1_{(\mathbb{C}^*)^{n-p'}}.
}
\]
The following observation forms the basis of the proof of Theorem \ref{Main3}.

\begin{lemma}\label{Compactness}
If $\text{rec}(\sigma) \in \Sigma$, then 
\[
\overline{\text{Log}^{-1}(\sigma)}^X \subseteq U_{\text{rec}(\sigma)}.
\]
\end{lemma}

\begin{proof}
Note that the isomorphism  $\varphi^1_\sigma$ restricts to the homeomorphism
$\overline{\pi_{\text{rec}(\sigma)}^{-1}(1)}^{U_{\text{rec}(\sigma)}} \simeq \overline{\mathbb{D}}^{p'}$, where $\overline{\mathbb{D}}$ is the closed unit disc in $\mathbb{C}$.
Write $\Phi$ for the action of $(S^1)^n$ on $U_{\text{rec}(\sigma)}$, and observe that
\[
\Phi\Bigg((S^1)^n \times \pi_{\text{rec}(\sigma)}^{-1}(1)\Bigg)=\bigcup_{x \in S_{N(\text{rec}(\sigma))}} \pi^{-1}_{\text{rec}(\sigma)}(x)
=\text{Log}^{-1}\big(\text{rec}(\sigma)^\circ\big).
\]
This shows that
\[
\Phi\Bigg((S^1)^n \times \overline{\pi_{\text{rec}(\sigma)}^{-1}(1)}^{U_{\text{rec}(\sigma)}}\Bigg)=\overline{\Phi\Big((S^1)^n \times \pi_{\text{rec}(\sigma)}^{-1}(1)\Big)}^{U_{\text{rec}(\sigma)}}=\overline{\text{Log}^{-1}\big(\text{rec}(\sigma)^\circ\big)}^{U_{\text{rec}(\sigma)}},
\]
where the compactness of $\overline{\pi_{\text{rec}(\sigma)}^{-1}(1)}^{U_{\text{rec}(\sigma)}}$ is used in the first equality.
Since the logarithm map is a submersion, the above implies
\[
\Phi\Bigg((S^1)^n \times \overline{\pi_{\text{rec}(\sigma)}^{-1}(1)}^{U_{\text{rec}(\sigma)}}\Bigg)=\overline{\text{Log}^{-1}\big(\text{rec}(\sigma)\big)}^{U_{\text{rec}(\sigma)}}.
\]
Therefore, the set on the right-hand side is compact. We use this to prove that $\overline{\text{Log}^{-1}(\sigma)}^{U_{\text{rec}(\sigma)}}$ is compact, and hence
\[
\overline{\text{Log}^{-1}(\sigma)}^{X}=\overline{\text{Log}^{-1}(\sigma)}^{U_{\text{rec}(\sigma)}} \subseteq U_{\text{rec}(\sigma)}.
\]
Let $\Delta$ be a bounded polyhedron in the Minkowski-Weyl decomposition
$\sigma=\Delta+\text{rec}(\sigma)$.
Write $\Psi$ for the action of $\mathbb{R}^n$ on $U_{\text{rec}(\sigma)}$, and observe that
\[
\Psi\Bigg(\Delta \times \text{Log}^{-1}\big(\text{rec}(\sigma)\big)\Bigg) 
=  \bigcup_{b \in \Delta} \text{Log}^{-1}\big(b+\text{rec}(\sigma)\big)= \text{Log}^{-1}\big(\sigma\big).
\]
This shows that
\[
\Psi\Bigg(\Delta \times \overline{\text{Log}^{-1}\big(\text{rec}(\sigma)\big)}^{U_{\text{rec}(\sigma)}}\Bigg)=\overline{\Psi\Big(\Delta \times \text{Log}^{-1}\big(\text{rec}(\sigma)\big)\Big)}^{U_{\text{rec}(\sigma)}}=\overline{\text{Log}^{-1}(\sigma)}^{U_{\text{rec}(\sigma)}},
\]
where the compactness of $\overline{\text{Log}^{-1}\big(\text{rec}(\sigma)\big)}^{U_{\text{rec}(\sigma)}}$ is used in the first equality.
Therefore, the set on the right-hand side is compact. 
\end{proof}

Let $\mathscr{C}$ be a $p$-dimensional balanced weighted complex in $\mathbb{R}^n$ with finitely many cells.

\begin{proposition}\label{ToricExtremality}
If $\mathscr{C}$ is non-degenerate, strongly extremal, and compatible with $\Sigma$, then $\overline{\mathscr{T}}_\mathscr{C}$ is a strongly extremal closed current on $X$.
\end{proposition}

\begin{proof}
Write  $D_\rho$ for the torus invariant prime divisor in $X$ corresponding to a $1$-dimensional cone $\rho$ in $\Sigma$. We note that, for any $p$-dimensional rational polyhedron $\sigma$ in $\mathscr{C}$, the subset
\[
D_\rho \cap \overline{\text{Log}^{-1}\big(\text{aff}(\sigma)\big)}^{U_{\text{rec}(\sigma)}} \subseteq U_{\text{rec}(\sigma)}
\] 
is either empty or a closed submanifold of Cauchy-Riemann dimension $p-1$. 
The subset is nonempty if and only if  $\text{rec}(\sigma)$ contains $\rho$, and in this case, for any $b \in \text{aff}(\sigma)$, we have the commutative diagram
\[
\small
\xymatrix{
D_\rho \cap \overline{\text{Log}^{-1}\big(\text{aff}(\sigma)\big)}^{U_{\text{rec}(\sigma)}} \ar[r] \ar[d]^\simeq&U_{\text{rec}(\sigma)} \ar[d]^{\varphi_\sigma^3}\\
\mathbb{C}^{p'-1} \times (\mathbb{C}^*)^{p-p'} \times (S^1)^{n-p} \ar[r]^{\qquad \quad e^{-b}}& \mathbb{C}^{p'} \times (\mathbb{C}^*)^{n-p'}.
}
\]

Let $\mathscr{T}$ be a closed current on $X$ with measure coefficients which has the same dimension and support as $\overline{\mathscr{T}}_\mathscr{C}$.
By Theorem \ref{Main2}, there is a complex number $c$ such that
\[
\mathscr{T}|_{(\mathbb{C}^*)^n}- c \cdot \mathscr{T}_\mathscr{C}=0.
\]
This implies that
\[
|\mathscr{T}-c \cdot \overline{\mathscr{T}}_\mathscr{C}| \subseteq \bigcup_{\rho,\sigma} \Bigg(D_\rho \cap \overline{\text{Log}^{-1}(\sigma)}^{X} \Bigg),
\]
where the union is over all pairs of $1$-dimensional cone $\rho$ in $\Sigma$ and $p$-dimensional cell $\sigma$ in $\mathscr{C}$. By Lemma \ref{Compactness}, we have
\[
|\mathscr{T}-c \cdot \overline{\mathscr{T}}_\mathscr{C}| \subseteq \bigcup_{\rho,\sigma}\Bigg(D_\rho \cap \overline{\text{Log}^{-1}\big(\text{aff}(\sigma)\big)}^{U_{\text{rec}(\sigma)}} \Bigg).
\]
The above commutative diagram shows that the right-hand side is a finite union of submanifolds of Cauchy-Riemann dimension $p-1$:
\[
 \bigcup_{\rho,\sigma}\Bigg(D_\rho \cap \overline{\text{Log}^{-1}\big(\text{aff}(\sigma)\big)}^{U_{\text{rec}(\sigma)}} \Bigg) \simeq  \bigcup_{\rho,\sigma}\Bigg(\mathbb{C}^{p'-1} \times (\mathbb{C}^*)^{p-p'} \times (S^1)^{n-p} \Bigg).
\]
By the first theorem on support \cite[Section III.2]{DemaillyBook1}, this implies
\[
\mathscr{T}-c \cdot \overline{\mathscr{T}}_\mathscr{C}=0.
\]
\end{proof}

\subsection{}

Let $D_1,\ldots,D_p$ be the torus invariant prime divisors in $X$ corresponding to distinct $1$-dimensional cones $\rho_1,\ldots,\rho_p$ in $\Sigma$. 
We fix a positive integer $l \le p$.

\begin{lemma}\label{Transversal}
Let $\sigma$ be a $p$-dimensional rational polyhedron in $\mathbb{R}^n$, $x \in S_{N(\sigma)}$, $b \in \text{aff}(\sigma)$.
\begin{enumerate}[(1)]
\item If $\text{rec}(\sigma) \in \Sigma$,  then $D_1,\ldots,D_{l}$ intersect transversally
 with the smooth subvariety
\[
\overline{ \pi_{\text{aff}(\sigma)}^{-1}(x) }^{U_{\text{rec}(\sigma)}} \subseteq U_{\text{rec}(\sigma)},
\]
and this intersection is nonempty if and only if $\text{rec}(\sigma)$ contains $\rho_1,\ldots,\rho_l$.
\item If $\text{rec}(\sigma) \in \Sigma$ and $\text{rec}(\sigma)$ contains $\rho_1,\ldots,\rho_p$, then 
\[
D_1 \cap \cdots \cap D_p \cap \overline{ \pi_{\text{aff}(\sigma)}^{-1}(x) }^{U_{\text{rec}(\sigma)}}  =\Big\{e^{-b} \cdot x  \cdot z_{\text{rec}(\sigma)}\Big\}. 
\]
\item If $\text{rec}(\sigma) \in \Sigma$ and $\text{rec}(\sigma)$ contains $\rho_1,\ldots,\rho_p$, then  the above intersection point is contained in the relative interior of 
\[
 \overline{ \pi_\sigma^{-1}(x) }^{U_{\text{rec}(\sigma)}} \subseteq \overline{ \pi_{\text{aff}(\sigma)}^{-1}(x) }^{U_{\text{rec}(\sigma)}}.
 \]
 \end{enumerate}
\end{lemma}

\begin{proof}
It is enough to prove the assertions when $x$ is the identity and $\sigma$ contains the origin. 
In this case, we have $\text{aff}(\sigma)=H_\sigma$ and $\text{rec}(\sigma) \subseteq \sigma$. 
If $\text{rec}(\sigma) \in \Sigma$, then $\text{rec}(\sigma)$ is unimodular, and there is a commutative diagram
\[
\small
\xymatrix{
\overline{\pi^{-1}_{\text{aff}(\sigma)}(1)}^{U_{\text{rec}(\sigma)}} \ar[d]^{\varphi^2_\sigma} \ar[r]& U_{\text{rec}(\sigma)} \ar[d]^{\varphi_\sigma^3}\\
\mathbb{C}^{p'}\times (\mathbb{C}^*)^{p-p'} \ar[r]&\mathbb{C}^{p'}\times (\mathbb{C}^*)^{n-p'}.
}
\]
If $\text{rec}(\sigma)$ does not contain $\rho_i$, then $D_i$ is disjoint from $U_{\text{rec}(\sigma)}$.
If  $\text{rec}(\sigma)$ contains $\rho_1,\ldots,\rho_l$, then 
\[
D_1 \cap \cdots \cap D_l \cap \overline{ \pi_{\text{aff}(\sigma)}^{-1}(x) }^{U_{\text{rec}(\sigma)}}   \simeq \mathbb{C}^{p'-l}\times (\mathbb{C}^*)^{p-p'}.
\]
If furthermore $l=p$, then $N(\sigma)=N(\text{rec}(\sigma))$, and the above intersection is the single point 
\[
\Big\{z_{\text{rec}(\sigma)} \Big\}\simeq \Big\{0_{\mathbb{C}^p} \times 1_{(\mathbb{C}^*)^{n-p}}\Big\}.
\]
This point is contained in the relative interior of
\[
\Bigg(\overline{\pi^{-1}_{\text{rec}(\sigma)}(1)}^{U_{\text{rec}(\sigma)}} \subseteq \overline{\pi^{-1}_{\text{aff}(\sigma)}(1)}^{U_{\text{rec}(\sigma)}} \Bigg) \simeq \Bigg(\overline{\mathbb{D}}^p \subseteq \mathbb{C}^p \Bigg),
\]
Since $\text{rec}(\sigma) \subseteq \sigma$, the point is contained in the relative interior of
\[
 \overline{ \pi_\sigma^{-1}(1) }^{U_{\text{rec}(\sigma)}} \subseteq \overline{ \pi_{\text{aff}(\sigma)}^{-1}(1) }^{U_{\text{rec}(\sigma)}}.
 \]
\end{proof}

The wedge product between positive closed currents will play an important role in the proof of Theorem \ref{Main3}. 
We briefly review the definition here, referring \cite{DemaillyBook1} and \cite{Dinh-Sibony} for details.
Write $d=\partial+\overline\partial$ for the usual decomposition of the exterior derivative on $X$, and set
\[
d^c:=\frac{1}{2\pi i}\big(\partial-\overline\partial\big).
\]
Let $u$ be a plurisubharmonic function on an open subset $U\subseteq X$, and let $\mathscr{T}$ be a positive closed current on $U$.
Since $\mathscr{T}$ has measure coefficients, $u\mathscr{T}$ is a well-defined current on $U$ if $u$ is locally integrable with respect to  $\text{tr}(\mathscr{T})$. In this case, we define
\[
dd^c(u) \wedge \mathscr{T} := dd^c(u\mathscr{T}).
\]
The wedge product is a positive closed current on $U$, and it vanishes identically when $u$ is pluriharmonic. 

Let $\mathscr{D}$ be a positive closed current on $U$ of degree $(1,1)$.
We define $\mathscr{D} \wedge \mathscr{T}$ as above, using open subsets $U_i \subseteq U$ covering $U$ and plurisubharmonic functions $u_i$ on $U_i$ satisfying
\[
 \mathscr{D}|_{U_i}=dd^cu_i.
\]
The wedge product does not depend on the choice of the open covering and local potentials, and it extends linearly to the case when $\mathscr{D}$ is \emph{almost positive}, that is, when $\mathscr{D}$ can be written as the sum of a positive closed current and a smooth current. If $\mathscr{D}_1,\ldots,\mathscr{D}_l$ are almost positive closed current on $U$ of degree $(1,1)$ satisfying the integrability condition, we define
\[
\mathscr{D}_1 \wedge \mathscr{D}_2 \wedge \ldots \wedge  \mathscr{D}_l \wedge \mathscr{T}:=\mathscr{D}_1 \wedge \Big( \mathscr{D}_2 \wedge \ldots \wedge  (\mathscr{D}_l \wedge \mathscr{T}) \Big).
\]

Let $\mathcal{C}$ be a $p$-dimensional tropical variety compatible with $\Sigma$, and let $\rho$ be a ray of $\Sigma$.
For each $p$-dimensional cell $\sigma$ of $\mathcal{C}$ whose recession cone contains $\rho$, we set
\[
\text{w}_{\text{star}(\rho,\mathcal{C})}(\overline{\sigma})=\text{w}_\mathcal{C}(\sigma),
\]
where $\overline{\sigma}$ is the image of $\sigma$ in the quotient space $N(\rho)_\mathbb{R}$.
This defines a $(p-1)$-dimensional tropical variety 
\[
\text{star}(\rho,\mathcal{C}) \subseteq N(\rho)_\mathbb{R},
\]
 whose $(p-1)$-dimensional cones correspond to 
$p$-dimensional cones of $\mathscr{C}$ whose recession cone contains $\rho$. 
For any $\sigma$ as above, the facets of $\overline{\sigma}$ are the images of the facets of $\sigma$ whose recession cone contains $\rho$,
and therefore the balancing condition for $\text{star}(\rho,\mathcal{C})$ follows from the balancing condition for $\mathscr{C}$.
The notation ``$\text{star}$'' is motivated by the important special case when $\mathcal{C}=\text{rec}(\mathcal{C})$.

\begin{proposition}\label{Wedge}
If $D_\rho$ is the torus invariant divisor of $X$ corresponding to a ray $\rho$ of $\Sigma$, then
\[
D_\rho \wedge \overline{\mathscr{T}}_\mathcal{C} = \overline{\mathscr{T}}_{\text{star}(\rho,\mathcal{C})}.
\]
\end{proposition}

Proposition \ref{Wedge} leads to an explicit description of the $0$-dimensional current $[D_{\rho_1}] \wedge \ldots \wedge [D_{\rho_p}] \wedge \overline{\mathscr{T}}_\mathscr{C}$
for distinct torus invariant divisors $D_{\rho_i}$.
For a $p$-dimensional rational polyhedron $\sigma$ compatible with $\Sigma$, we write $\mu_\sigma$ be the normalized Haar measure on $S_{N(\sigma)}$.
If $\text{rec}(\sigma)$ is $p$-dimensional, we define a closed embedding 
\[
\iota_\sigma: S_{N(\sigma)} \longrightarrow X, \qquad t \longmapsto e^{-b} \cdot  t \cdot z_{\text{rec}(\sigma)}, \qquad b \in \text{aff}(\sigma),
\]
which does not depend on the choice of $b$.
Repeated application of Proposition \ref{Wedge} gives
\[
[D_{\rho_1}] \wedge \ldots \wedge [D_{\rho_p}] \wedge \overline{\mathscr{T}}_\mathscr{C}= \sum_\sigma \text{w}_\mathscr{C}(\sigma) \iota_{\sigma*} (d\mu_\sigma),
\]
where the sum is over all $p$-dimensional cells $\sigma$ in $\mathscr{C}$ such that $\text{rec}(\sigma)=\text{cone}(\rho_1,\ldots,\rho_p)$.

\begin{proof}
We first note that the support of $D_\rho \wedge \overline{\mathscr{T}}_\mathcal{C}$ is contained in $D_\rho$. Indeed, we have
\[
dd^c\Big( \text{log} |f| \ \overline{\mathscr{T}}_\mathcal{C}|_U\Big)=dd^c\Big(\text{log} |f|\Big)\wedge \overline{\mathscr{T}}_\mathcal{C}|_U=0 \wedge \overline{\mathscr{T}}_\mathcal{C}|_U=0
\]
for any nonvanishing holomorphic function $f$ on an open subset $U \subseteq X$.

Let $\sigma$ be a $p$-dimensional cone of $\mathcal{C}$.
If the recession cone of $\sigma$ contains $\rho$, then there is a natural isomorphism
\[
S_{N(\sigma)} \simeq S_{N(\overline{\sigma})}.
\]
Using the above identification, one can check in toric  local coordinates for $U_{\text{rec}(\sigma)}$ in Section \ref{ToricLocalCoordinates} that, for any $x$ in $S_{N(\sigma)}$,
\[
D_\rho \cap \overline{\pi_\sigma^{-1}(x)}=\left\{\begin{array}{cl}\overline{\pi^{-1}_{\overline{\sigma}}(x)} & \text{if the recession cone of $\sigma$ contains $\rho$,} \\ \varnothing & \text{if the recession cone of $\sigma$ does not contain $\rho$.}\end{array} \right.
\]
Suppose that the recession cone of $\sigma$ contains $\rho$. 
If $f_\rho$ is a defining equation of $D_\rho$ on an open subset $U \subseteq X$,
an application of the Poincar\'e-Lelong formula shows that
\[
dd^c\Big(\text{log}|f_\rho| \ [\overline{\pi_\sigma^{-1}(x)}]|_U\Big)=[\overline{\pi_{\overline{\sigma}}^{-1}(x)}]|_U+\mathscr{R}_{\sigma}(x),
\]
where $\mathscr{R}_\sigma(x)$ is a current whose support is contained in the boundary of  $\overline{\pi_\sigma^{-1}(x)}$.
It follows that
\[
dd^c\Big(\text{log}|f_\rho|\ \overline{\mathscr{T}}_\sigma|_U\Big)=\overline{\mathscr{T}}_{\overline{\sigma}}|_U+\mathscr{R}_\sigma,
\]
where $\mathscr{R}_\sigma$ is a current whose support is contained in the boundary of $\overline{\text{Log}^{-1}(\sigma)}$.
Therefore, the support of the closed current
\[
D_\rho \wedge \overline{\mathscr{T}}_\mathcal{C} - \overline{\mathscr{T}}_{\text{star}(\rho,\mathcal{C})}
\]
 is contained in the intersection of $D_{\rho}$ and the piecewise smooth manifold
 $\bigcup_\sigma\  \partial\ \overline{\text{Log}^{-1}(\sigma)}$,
 where the union is over all $p$-dimensional cells of $\mathcal{C}$ whose recession contains $\rho$.
 The intersection in question is the union of closed submanifolds of Cauchy-Riemann dimension $p-2$, and hence
\[
D_\rho \wedge \overline{\mathscr{T}}_\mathcal{C} - \overline{\mathscr{T}}_{\text{star}(\rho,\mathcal{C})}=0.
\]
\end{proof}

We illustrate the argument in coordinates in the  representative case when $\sigma$ is a cone containing $\rho$. 
Consider the toric coordinate system on  $U_\sigma$ with
\[
U_\sigma \simeq \mathbb{C} \times \mathbb{C}^{p-1} \times (\mathbb{C}^*)^{n-p} \quad \text{and} \quad D_\rho |_{U_\sigma} \simeq 0 \times \mathbb{C}^{p-1} \times  (\mathbb{C}^*)^{n-p}.
\]
Writing $\overline{\mathbb{D}}$ for the closed unit disc in $\mathbb{C}$, we have
\[
\overline{\pi^{-1}_{\sigma}(x)} \simeq  \overline{\mathbb{D}}^p \times x \quad  \text{and} \quad
\overline{\text{Log}^{-1}(\sigma)} \simeq  \overline{\mathbb{D}}^p \times (S^1)^{n-p}.
\]
The boundary of the latter has $p$ components of the form
\[
\overline{\mathbb{D}} \times \cdots \times \overline{\mathbb{D}} \times S^1 \times \overline{\mathbb{D}} \times \cdots \times \overline{\mathbb{D}} \times (S^1)^{n-p},
\]
whose intersection with $D_\rho$ has Cauchy-Riemann dimension $p-2$.
Therefore the intersection cannot support any normal closed current of dimension $(p-1,p-1)$.
 
\subsection{}

We begin the proof of Theorem \ref{Main3}. Fix a positive integer $l \le p$,  and let
\begin{eqnarray*}
\rho_1,\ldots,\rho_p &:=& \text{distinct $1$-dimensional cones in $\Sigma$},\\
D_1,\ldots,D_p&:=& \text{torus invariant divisors of $\rho_1,\ldots,\rho_p$},\\
L_1,\ldots,L_p&:=& \text{hermitian line bundles on $X$ corresponding to $D_1,\ldots,D_p$},\\
w_1,\ldots,w_p&:=& \text{Chern forms of the line bundles $L_1,\ldots,L_p$}.
\end{eqnarray*}
If  $s_i$ is a holomorphic section of $\mathcal{O}_X(D_i)$ that defines $D_i$, then
the Poincar\'e-Lelong formula says that 
\[
dd^c \log|s_i|=[D_i]-w_i.
\]

\begin{proposition}\label{CohomologyClass}
If $\mathscr{C}$ is a $p$-dimensional tropical variety compatible with $\Sigma$, then
\[
\Big\{[D_1] \wedge \ldots \wedge [D_l] \wedge \overline{\mathscr{T}}_\mathscr{C}\Big\}=\{w_1\} \wedge \ldots \wedge \{w_l\} \wedge \big\{\overline{\mathscr{T}}_\mathscr{C}\big\}.
\]
\end{proposition}

\begin{proof}
The statement follows from repeated application of the following general fact. Let $\mathscr{T}$ be a positive closed current on $X$, and let $\mathscr{D}$ be a positive closed $(1,1)$-current on $X$. We write
\[
\mathscr{D}=w+dd^cu,
\] 
where $w$ is a smooth $(1,1)$-form and $u$ is an \emph{almost plurisubharmonic function}, a function
that is locally equal to the sum of a smooth function and a plurisubharmonic function.
The general fact to be applied is:
\[
\text{If $u$ is locally integrable with respect to $\text{tr}(\mathscr{T})$, then  $\{\mathscr{D} \wedge \mathscr{T}\}=\{w\} \wedge \{\mathscr{T}\}$.}
\]
To see this, we use Demailly's regularization theorem \cite{DemaillyRegularization} to construct a sequence of smooth functions $u_j$ decreasing to $u$ and a smooth positive closed $(1,1)$-form $\psi$ such that
 \[
 \psi + dd^c u_j \ge 0.
 \]
Choose open subsets $U_i \subseteq X$ covering $X$ and plurisubharmonic functions $\varphi_i$ on $U_i$ such that
\[
\psi|_{U_i}=dd^c \varphi_i.
\]
Then $\varphi_i+(u_j|_{U_i})$ is a sequence of plurisubharmonic functions on $U_i$ decreasing to $\varphi_i+(u|_{U_i})$.
By the monotone convergence theorem for wedge products \cite{Dinh-Sibony}, we have
\[
\lim_{j \to \infty} (\psi+dd^c u_j)|_{U_i} \wedge \mathscr{T}|_{U_i}=(\psi+dd^c u)|_{U_i} \wedge \mathscr{T}|_{U_i}.
\]
Since $X$ is compact, the open covering of $X$ can be taken to be finite, and hence 
\[
\lim_{j \to \infty} (\psi+dd^c u_j) \wedge \mathscr{T}=(\psi+dd^c u) \wedge \mathscr{T}.
\]
By continuity of the cohomology class assignment, this implies
\[
\lim_{j \to \infty} \Big\{(w+dd^c u_j) \wedge \mathscr{T}\Big\}=\Big\{(w+dd^c u) \wedge \mathscr{T} \Big\}=\{ \mathscr{D} \wedge \mathscr{T}\}.
\]
Since $w+dd^c u_j$ is smooth and $dd^c u_j$ is exact, the left-hand side is equal to
\[
\lim_{j \to \infty} \big\{w+dd^c u_j\big\} \wedge \{\mathscr{T}\}=\{w\} \wedge \{\mathscr{T}\}.
\]
\end{proof}

\begin{proof}[Proof of Theorem \ref{Main3}]
Suppose $\rho_1,\ldots,\rho_p$ generates a $p$-dimensional cone $\gamma$ in $\Sigma$. Since $X$ is smooth, every $p$-dimensional cone of $\Sigma$ is of this form.
The torus orbit closure $V(\gamma) \subseteq X$ corresponding to $\gamma$ is the transversal intersection of $D_1,\ldots,D_p$,
and its fundamental class is Poincar\'e dual to $\{w_1\} \wedge \ldots \wedge \{w_p\}$.
We show that
\[
 \langle \overline{\mathscr{T}}_\mathscr{C},w_1 \wedge \ldots \wedge w_p\rangle= \int_X w_1 \wedge \ldots \wedge w_p \wedge \overline{\mathscr{T}}_\mathscr{C}=\text{w}_{\text{rec}(\mathscr{C})}(\gamma).
 \]
By Proposition \ref{CohomologyClass}, we have
\[
\int_X w_1 \wedge \ldots \wedge w_p \wedge \overline{\mathscr{T}}_\mathscr{C}=\int_X [D_1] \wedge \ldots [D_p] \wedge \overline{\mathscr{T}}_\mathscr{C}.
\]
By Proposition \ref{Wedge}, the right-hand side is 
\[
\int_X [D_1] \wedge \ldots [D_p] \wedge \overline{\mathscr{T}}_\mathscr{C}=\int_X \Bigg(\sum_\sigma \text{w}_\mathscr{C}(\sigma)\iota_{\sigma *}(d\mu_\sigma)\Bigg)=\sum_\sigma \text{w}_\mathscr{C}(\sigma),
 \]
where the sum is over all $p$-dimensional cells $\sigma$ in $\mathscr{C}$ such that $\text{rec}(\sigma)=\gamma$. Note that the sum is, by definition of the recession of $\mathscr{C}$, the weight $\text{w}_{\text{rec}(\mathscr{C})}(\gamma)$. Since the above computation is valid for each $p$-dimensional cone $\gamma$ in $\Sigma$,
 we have
 \[
\{\overline{\mathscr{T}}_\mathscr{C}\}=\text{rec}(\mathscr{C}) \in H^{q,q}(X).
\]
\end{proof}

\section{The strongly positive Hodge conjecture}\label{SectionLast}

\subsection{}

This section is devoted to the construction of the following example.

\begin{theorem}\label{Main4}
There is a $4$-dimensional smooth projective variety $X$ and a $(2,2)$-dimensional strongly positive closed current $\mathscr{T}$ on $X$ with the following properties:
\begin{enumerate}[(1)]
\item The cohomology class of $\mathscr{T}$ satisfies
\[
\{\mathscr{T}\} \in H^4(X,\mathbb{Z}) \cap H^{2,2}(X,\mathbb{C}).
\]
\item The current $\mathscr{T}$ is not a weak limit of the form
\[
\lim_{i \to \infty} \mathscr{T}_i, \quad \mathscr{T}_i=\sum_j \lambda_{ij} [Z_{ij}],
\]
where $\lambda_{ij}$ are positive real numbers and $Z_{ij}$ are irreducible surfaces in $X$.
\end{enumerate}
\end{theorem}

The above $X$ and $\mathscr{T}$ have other notable properties: $X$ is a toric variety, $\mathscr{T}$ is strongly extremal, and $\{\mathscr{T}\}$ generates an extremal ray of the nef cone of $X$ (the dual of the effective cone of complementary dimension with respect to the Poincar\'e pairing). 
It follows from \cite[Corollary 4.6]{Fulton-Sturmfels}
that any nef class in a smooth complete toric variety is effective, and hence there are nonnegative integers $\lambda_j$ and irreducible surfaces $Z_j\subseteq X$ such that
\[
\big\{\mathscr{T}\big\}=\sum_j \lambda_j \big\{[Z_j]\big\}.
\]
This example shows that, in general, $\textrm{HC}^+$ is not true and not implied by $\textrm{HC}'$. 

\begin{remark}
Assume that $\mathscr{C}$ is a strongly extremal tropical variety which is approximable as a set by logarithmic limit sets of a family of closed algebraic subvarieties of $(\mathbb{C}^*)^n.$ Then by \cite[Theorem 5.2.7]{Babaee}, there are closed subvarieties $Z_i \subseteq (\mathbb{C}^*)^n$, and positive real numbers $\lambda_i$ such that 
\[
\mathscr{T}_{\mathscr{C}}=\lim_{i \to \infty} \lambda_i [Z_i].
\]
Therefore, non-approximability of the tropical current $\mathscr{T}= \overline{\mathscr{T}}_{\mathscr{C}}$ in  Theorem \ref{Main4}, implies that there is no family of algebraic subvarieties of $(\mathbb{C}^*)^n,$ whose logarithmic limit sets approximate its underlying tropical variety $\mathscr{C}$ as a set. 

\end{remark}


\subsection{}

Let $G$ be an edge-weighted geometric graph in $\mathbb{R}^n \setminus \{0\}$,
 that is, an edge-weighted graph whose vertices are nonzero vectors in $\mathbb{R}^n$ and edges are line segments in $\mathbb{R}^n \setminus \{0\}$. We suppose that all the edge-weights are real numbers.
We  write 
\begin{eqnarray*}
u_1,u_2,\ldots &:=& \text{the vertices of $G$},\\
u_iu_j &:=& \text{the edge connecting $u_i$ and $u_j$},\\
w_{ij} &:=& \text{the weight on the edge $u_iu_j$}.
\end{eqnarray*}
We say that an edge $u_iu_j$  is \emph{positive} or \emph{negative} according to the sign of the weight $w_{ij}$.

\begin{definition}
An edge-weighted geometric graph $G \subseteq \mathbb{R}^n \setminus \{0\}$ satisfies the \emph{balancing condition} at its vertex $u_i$ if there is a real number $d_i$ such that
\[
d_iu_i=\sum_{u_i \sim u_j} w_{ij} u_j,
\]
where the sum is over all neighbors of $u_i$ in $G$. The graph $G$ is \emph{balanced} if it satisfies the balancing condition at each of its vertices.
\end{definition}

The real numbers $d_i$ are uniquely determined by $G$ because all the vertices of $G$ are nonzero.
When $G$ is balanced, we define the \emph{tropical Laplacian} of $G$ to be the real symmetric matrix
$L_G$ with entries
\[
(L_G)_{ij}:=\left\{\begin{array}{cl} d_i & \text{if $u_i=u_j$,} \\ -w_{ij} & \text{if $u_i \sim u_j$,} \\ 0 & \text{if otherwise,}\end{array}\right.
\]
where the diagonal entries $d_i$ are the real numbers satisfying
\[
d_i u_i = \sum_{u_i \sim u_j} w_{ij} u_j.
\]
The tropical Laplacian of $G$ has a combinatorial part and a geometric part:
\[
L_G=L(G)-D(G).
\]
The combinatorial part $L(G)$ is the combinatorial Laplacian of the abstract graph of $G$ as defined in \cite{Chung}, and the geometric part $D(G)$ is a diagonal matrix that depends on the position of the vertices of $G$.

\begin{definition}
When $G$ is balanced, we define the \emph{signature} of $G$ to be the triple
\begin{eqnarray*}
n_+(G) &:=& \text{the positive index of inertia of $L_G$},\\
n_-(G) &:=& \text{the negative index of inertia of $L_G$},\\
n_0(G)\ &:=& \text{the corank of $L_G$}.
\end{eqnarray*}
\end{definition}

Let $F$ be a $2$-dimensional real weighted fan in $\mathbb{R}^n$,
that is, a $2$-dimensional weighted complex all of whose $2$-dimensional cells are cones with real weights.
We define an edge-weighted geometric graph $G(F) \subseteq \mathbb{R}^n \setminus \{0\}$ as follows:
\begin{enumerate}[(1)]
\item The set of vertices of $G(F)$ is 
\[
\big\{u_i \mid \text{$u_i$ is a primitive generator of a $1$-dimensional cone in $F$}\big\}.
\]
\item The set of edges of $G(F)$ is
\[
\big\{u_iu_j \mid \text{the cone over $u_iu_j$ is a $2$-dimensional cone in $F$ with nonzero weight}\big\}.
\]
\item The weights on the edges of $G(F)$ are 
\[
w_{ij}:=\text{w}_F\big(\text{cone}(u_iu_j)\big).
\]
\end{enumerate}

We say that a weighted fan is \emph{unimodular} if all of its cones are unimodular.

\begin{proposition}
When $F$ is unimodular, $F$ is balanced if and only if $G(F)$ is balanced.
\end{proposition}

\begin{proof}
Let $u_iu_j$ be an edge of $G(F)$, let $\sigma$ be the cone over $u_iu_j$, and let $\tau$ be the cone over $u_i$.
Since $\sigma$ is unimodular, the image of $u_j$
in the normal lattice of $\tau$ is $u_{\sigma/\tau}$,
the primitive generator of the ray
\[
\text{cone}(\sigma-\tau)/H_\tau \subseteq H_\sigma/H_\tau.
\]
Therefore, the balancing condition for $F$ at $\tau$ is equivalent to the condition
\[
\sum_{u_i \sim u_j} w_{ij} u_j  \in \mathbb{R} \cdot u_i.
\]
\end{proof}

A geometric graph $G \subseteq \mathbb{R}^n$ is said to be \emph{locally extremal} if the set of neighbors of $u_i$ is linearly independent for every vertex $u_i \in G$.

\begin{proposition}\label{GraphExtremality}
Let $F$ be a $2$-dimensional real balanced weighted fan in $\mathbb{R}^n$. If  $G(F)$ is connected and locally extremal, then $F$ is strongly extremal.
\end{proposition}

\begin{proof}
Let $u_iu_j$ be an edge of $G(F)$, let $\sigma$ be the cone over $u_iu_j$, and let $\tau$ be the cone over $u_i$.
The image of $u_j$ in the normal lattice of $\tau$ is a nonzero multiple of $u_{\sigma/\tau}$, and hence the weighted fan $F$ is locally extremal if and only if $G(F)$ is locally extremal.
Since $F$ is pure dimensional, $F$ is connected in codimension $1$ if and only if $G(F)$ is connected, and therefore the assertion follows from Proposition \ref{LocalGlobalExtremality}.
\end{proof}

When $F$ is balanced and unimodular, we define the \emph{tropical Laplacian} of $F$ to be the tropical Laplacian of $G(F)$, and the \emph{signature} of $F$ to be the signature of $G(F)$:
\[
\Big(n_+(F),n_-(F),n_0(F)\Big) := \Big(n_+\big(G(F)\big),n_-\big(G(F)\big),n_0\big(G(F)\big)\Big).
\]
In Sections \ref{PlusConstruction} and \ref{MinusConstruction}, we  introduce two basic operations on $F$ and analyze the change of the signature of $F$ under each operation
(see Figure \ref{fig:basic-op}).

\begin{figure}[bt]
\begin{center}
 \includegraphics[width=8cm]{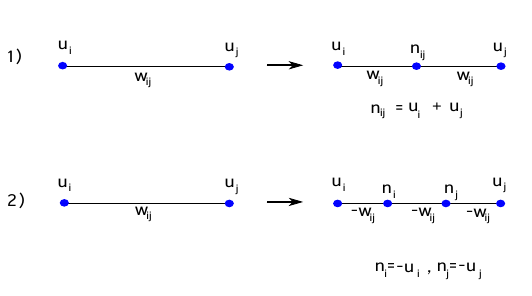}
\end{center}
\caption{The operation $F \longmapsto F_{ij}^+$ produces one new eigenvalue whose sign coincides with the sign of $w_{ij}$, and the operation $F \longmapsto F_{ij}^-$ produces one new positive and one new negative eigenvalue.}
\label{fig:basic-op}
\end{figure} 

\subsection{}\label{PlusConstruction}

Let $F$ be a $2$-dimensional real weighted fan in $\mathbb{R}^n$, and suppose that $u_1u_2$ is an edge of $G(F)$. We set
\[
n_{12}:=u_1+u_2,
\]
and define a $2$-dimensional real weighted fan $F^+_{12}$ as follows:
\begin{enumerate}[(1)]
\item The set of $1$-dimensional cones in $F_{12}^+$  is
\[
\Big\{\text{cone}(n_{12})\Big\} \cup \Big\{\text{$1$-dimensional cones in $F$}\Big\}.
\]
\item The set of $2$-dimensional cones in $F^+_{12}$ is
\[
\Big\{\text{cone}(u_1n_{12}), \text{cone}(u_2n_{12})\Big\} \cup \Big\{ \text{$2$-dimensional cones in $F$ other than $\text{cone}(u_1u_2)$}\Big\}.
\]
\item The weights on the $2$-dimensional cones in $F^+_{12}$ are
\begin{eqnarray*}
&&\text{w}_{F^+_{12}}\big(\text{cone}(u_1n_{12})\big):=w_{12},\\
&&\text{w}_{F^+_{12}}\big(\text{cone}(u_2n_{12})\big):=w_{12},\\
&&\text{w}_{F^+_{12}}\big(\text{cone}(u_iu_j)\big):=w_{ij}.
\end{eqnarray*}
\end{enumerate}
The abstract graph of $G(F^+_{12})$ is a subdivision of the abstract graph of $G(F)$, with one new vertex $n_{12}$ subdividing the edge connecting $u_1$ and $u_2$.
It is easy to see that
\begin{enumerate}[--]
\item $F_{12}^+$ is balanced if and only if $F$ is balanced, 
\vspace{1mm}
\item $F^+_{12}$ is unimodular if and only if $F$ is unimodular, 
\vspace{1mm}
\item $F_{12}^+$ is non-degenerate if and only if $F$ is non-degenerate, 
\vspace{1mm}
\item $F^+_{12}$ is strongly extremal if and only if $F$ is strongly extremal.
\end{enumerate}

Suppose $F$ is balanced and unimodular, so that the tropical Laplacians $L_{G(F)}$ and $L_{G(F_{12}^+)}$ are defined. The balancing condition for $G(F)$ translates to the balancing condition for $G(F^+_{12})$, and we can compute the diagonal entries of  $L_{G(F_{12}^+)}$ from the diagonal entries of $L_{G(F)}$. 
The balancing condition for $G(F)$ at $u_i$ is
\[
d_iu_i=\sum_{u_i \sim u_j} w_{ij}u_j,
\]
where the sum is over all neighbors of $u_i$ in the graph $G(F)$.
\begin{enumerate}[(1)]
\item The balancing condition for $G(F_{12}^+)$ at $u_1$ is
\[
(d_1+w_{12})u_1=w_{12}n_{12}+\sum_{u_j} w_{1j}u_j,
\]
where the sum is over all neighbors of $u_1$  other than $u_2$ in  the graph $G(F)$.
\item The balancing condition for $G(F_{12}^+)$ at $u_2$ is
\[
(d_2+w_{12})u_2=w_{12}n_{12}+\sum_{u_j} w_{2j}u_j,
\]
where the sum is over all neighbors of $u_2$  other than $u_1$ in the graph $G(F)$.
\item The balancing condition for $G(F_{12}^+)$ at $n_{12}$ is
\[
w_{12}n_{12}=w_{12}u_1+w_{12}u_2.
\]
\end{enumerate}
Around any other vertex,  the geometric graphs $G(F)$ and $G(F^+_{12})$ are identical, and hence the diagonal entries of the two tropical Laplacians agree.

\begin{proposition}\label{PositiveSignature}
We have
\[
\Big(n_+(F^+_{12}), \ n_-(F^+_{12}), \ n_0(F^+_{12})\Big)=\left\{\begin{array}{cc}\Big(n_+(F)+1,\ n_-(F),\ n_0(F)\Big) &\text{if $w_{12}$ is positive,} \vspace{2mm}\\ \Big(n_+(F),\ n_-(F)+1,\ n_0(F)\Big) &\text{if $w_{12}$ is negative.}  \end{array}\right.
\]
\end{proposition}

\begin{proof}
Let $Q_{G(F)}$ and $Q_{G(F^+_{12})}$ be the quadratic forms associated to  $L_{G(F)}$ and $L_{G(F^+_{12})}$ respectively.
The above analysis of the balancing condition for $G(F^+_{12})$ shows that
\begin{multline*}
Q_{G(F^+_{12})}(y_{12},x_1,x_2,x_3,\ldots)-Q_{G(F)}(x_1,x_2,x_3,\ldots) =\\
w_{12} x_1^2 +w_{12}  x_2^2+2w_{12} x_1x_2+w_{12}  y_{12}^2-2w_{12}  x_1y_{12}-2y_{12} x_{2}y_{12},
\end{multline*}
where $y_{12}$ is the variable corresponding to the new vertex $n_{12}$ and $x_i$ is the variable corresponding to $u_i$. The above equation simplifies to
\[
Q_{G(F^+_{12})}(y_{12},x_1,x_2,x_3,\ldots)-Q_{G(F)}(x_1,x_2,x_3,\ldots)=w_{12} \big(y_{12}-x_1-x_2\big)^2,
\]
and the conclusion follows.
\end{proof}

\subsection{}\label{MinusConstruction}

An edge $u_1u_2$ of a geometric graph in $\mathbb{R}^n \setminus \{0\}$ is said to be in \emph{general position} if for every other edge $u_iu_j$ of the graph
\[
\text{span}(u_1,u_2) \cap \text{span}(u_i,u_j)=\left\{\begin{array}{cl} 0 &\text{if $\{u_1,u_2\} \cap \{u_i,u_j\}=\emptyset$,} \\ \text{span}(u_1)& \text{if $\{u_1,u_2\} \cap \{u_i,u_j\}=\{u_1\}$,} \\ \text{span}(u_2) &\text{if $\{u_1,u_2\} \cap \{u_i,u_j\}=\{u_2\}$.}\end{array}\right.
\]
Let $F$ be a $2$-dimensional real weighted fan in $\mathbb{R}^n$, and suppose that $u_1u_2$ is an edge of $G(F)$. If $u_1u_2$ is in general position, we set
\[
n_1:=-u_1, \quad n_2:=-u_2,
\] 
and  define a $2$-dimensional real weighted fan $F^-_{12}$ as follows:
\begin{enumerate}[(1)]
\item The set of $1$-dimensional cones in $F^-_{12}$ is
\[
\Big\{\text{cone}(n_1), \text{cone}(n_2)\Big\} \cup \Big\{\text{$1$-dimensional cones in $F$}\Big\}.
\]
\item The set of $2$-dimensional cones in $F^-_{12}$ is
\[
\Big\{\text{cone}(n_1n_2),\text{cone}(n_1u_2),\text{cone}(u_1n_2)\Big\} \cup \Big\{\text{$2$-dimensional cones in $F$ other than $\text{cone}(u_1u_2)$}\Big\}.
\]
\item The weights on the $2$-dimensional cones in $F^-_{12}$ are
\begin{eqnarray*}
&&\text{w}_{F^-_{12}}\big(\text{cone}(n_1n_2)\big):=-w_{12},\\
&&\text{w}_{F^-_{12}}\big(\text{cone}(n_1u_2)\big):=-w_{12},\\
&&\text{w}_{F^-_{12}}\big(\text{cone}(u_1n_2)\big):=-w_{12},\\
&&\text{w}_{F^-_{12}}\big(\text{cone}(u_iu_j)\big):=w_{ij}.
\end{eqnarray*}
\end{enumerate}
The cones in $F_{12}^-$ form a fan because $u_1u_2$ is in general position. The abstract graph of $G(F^-_{12})$ is a subdivision of the abstract graph of $G(F)$, with two new vertices $n_1$ and $n_2$ subdividing the edge connecting $u_1$ and $u_2$.
It is easy to see that
\begin{enumerate}[--]
\item $F^-_{12}$ is balanced if and only if $F$ is balanced,
\vspace{1mm}
\item $F^-_{12}$ is unimodular if and only if $F$ is unimodular, 
\vspace{1mm}
\item $F_{12}^-$ is non-degenerate if and only if $F$ is non-degenerate, 
\vspace{1mm}
\item $F^-_{12}$ is strongly extremal if and only if $F$ is strongly extremal.
\end{enumerate}

Suppose $F$ is balanced and unimodular, so that the tropical Laplacians $L_{G(F)}$ and $L_{G(F_{12}^-)}$ are defined. The balancing condition for $G(F)$ translates to the balancing condition for $G(F^-_{12})$, and we can compute the diagonal entries of $L_{G(F^-_{12})}$ from the diagonal entries of $L_{G(F)}$. 
The balancing condition for $G(F)$ at $u_i$ is
\[
d_iu_i=\sum_{u_i \sim u_j} w_{ij}u_j,
\]
where the sum is over all neighbors of $u_i$ in the graph $G(F)$.

\begin{enumerate}[(1)]
\item The balancing condition for $G(F_{12}^-)$ at $u_1$ is
\[
d_1u_1=(-w_{12})n_2 + \sum_j w_{1j}u_j,
\]
where the sum is over all neighbors of $u_1$ other than $u_2$ in the graph $G(F)$.
\item The balancing condition for $G(F_{12}^-)$ at $u_2$ is
\[
d_2u_2=(-w_{12})n_1 + \sum_j w_{2j}u_j,
\]
where the sum is over all neighbors of $u_2$ other than $u_1$ in the graph $G(F)$.
\item The balancing condition for $G(F_{12}^-)$ at $n_1$ is
\[
0 \cdot n_{1}=(-w_{12})u_2+(-w_{12})n_2.
\]
\item The balancing condition for $G(F_{12}^-)$ at $n_2$ is
\[
0 \cdot n_{2}=(-w_{12})u_1+(-w_{12})n_1.
\]
\end{enumerate}
Around any other vertex, the geometric graphs $G(F)$ and $G(F^-_{12})$ are identical, and hence the diagonal entries of the two tropical Laplacians agree.

\begin{proposition}\label{MinusSignature}
We have
\[
\Big(n_+(F^-_{12}),\ n_-(F^-_{12}),\ n_0(F^-_{12})\Big)=\Big(n_+(F)+1,\ n_-(F)+1,\ n_0(F) \Big).
\]
\end{proposition}

\begin{proof}
Let $Q_{G(F)}$ and $Q_{G(F^-_{12})}$ be the quadratic forms associated to $L_{G(F)}$ and $L_{G(F^-_{12})}$ respectively.
The above analysis of the balancing condition for $G(F^-_{12})$ shows that
\begin{multline*}
Q_{G(F^-_{12})}(y_{1},y_2,x_1,x_2,x_3,\ldots)-Q_{G(F)}(x_1,x_2,x_3,\ldots)=\\
w_{12} x_1x_2+w_{12} x_1y_2+w_{12}  x_2y_1+w_{12}  y_1y_2,
\end{multline*}
where $y_1,y_2$ are variables corresponding to $n_1,n_2$ respectively and $x_i$ is the variable corresponding to $u_i$. The above equation simplifies to
\[
Q_{G(F^-_{12})}(y_{1},y_2,x_1,x_2,x_3,\ldots)-Q_{G(F)}(x_1,x_2,x_3,\ldots)=w_{12}(y_1+x_1)(y_2+x_2),
\]
and the conclusion follows.

\end{proof}

\subsection{}

Let $X$ be an $n$-dimensional smooth projective toric variety, $\Sigma$ be the fan of $X$, and let $p$ be an integer $\ge 2$. 

\begin{proposition}\label{OneNegativeEigenvalue}
Let $\{\mathscr{T}\}$ be a $(p,p)$-dimensional cohomology class in $X$.
If there are nonnegative real numbers $\lambda_i$ and $p$-dimensional irreducible subvarieties $Z_i \subseteq X$ such that
\[
\{\mathscr{T}\}=\lim_{i \to \infty} \big\{\lambda_i [Z_i] \big\},
\]
then, for any nef divisors $H_1,H_2,\ldots$ on $X$, the tropical Laplacian of the $2$-dimensional balanced weighted fan
\[
\{H_1\} \cup \ldots \cup \{H_{p-2}\} \cup \{\mathscr{T}\}  
\]
has at most one negative eigenvalue.
\end{proposition}

In particular,  by continuity of the cohomology class assignment, if a $(2,2)$-dimensional closed current $\mathscr{T}$ on $X$ is the weak limit of the form
\[
\mathscr{T}=\lim_{i \to \infty} \lambda_i [Z_i],
\]
where  $\lambda_i$ are nonnegative real numbers and $Z_i$ are irreducible surfaces in $X$, then the tropical Laplacian of  $\{\mathscr{T}\}$ has at most one negative eigenvalue.

\begin{proof}
Repeatedly applying the Bertini theorem \cite[Corollary 6.11]{Jouanolou} to a  general element of the linear system $|H_i|$, we are reduced to the case when $p=2$: 
If there are nonnegative real numbers $\lambda_i$ and irreducible surfaces $Z_i \subseteq X$ such that
\[
\{\mathscr{T}\}=\lim_{i \to \infty} \big\{\lambda_i [Z_i] \big\},
\]
then the tropical Laplacian of $\{\mathscr{T}\}$ has at most one negative eigenvalue.
By continuity,  it is enough to prove the following statement: If $Z \subseteq X$ is an irreducible surface, then the tropical Laplacian of $\big\{[Z]\big\}$ has exactly one negative eigenvalue.

 Let $F$ be the cohomology class $\big\{[Z]\big\}$, viewed as a $2$-dimensional weighted fan in $\mathbb{R}^n$, and let
 \begin{eqnarray*}
 u_1,u_2,\ldots &:=& \text{primitive generators of $1$-dimensional cones in $\Sigma$,}\\
 D_1,D_2,\ldots &:=& \text{torus-invariant prime divisors of $\text{cone}(u_1),\text{cone}(u_2),\ldots$,}\\
 L_1,L_2,\ldots &:=& \text{line bundles on $X$ corresponding to $D_1,D_2,\ldots$.}
 \end{eqnarray*}
The $2$-dimensional cones in $F$ are the $2$-dimensional cones in $\Sigma$, and the weights are given by
\[
w_{ij}:=\text{w}_F\big(\text{cone}(u_iu_j)\big)=c_1(L_i) \cup c_1(L_j) \cap \text{cl}(Z)=D_i \cdot D_j \cdot \text{cl}(Z).
\] 
Let $d_i$ be a diagonal entry of the tropical Laplacian of $F$, determined by the balancing condition
\[
d_iu_i=\sum_{u_i \sim u_j} w_{ij}u_j,
\]
where the sum is over all neighbors of $u_i$ in the graph of $F$. We claim that
\[
d_i=-D_i \cdot D_i \cdot \text{cl}(Z).
\]
To see this, choose any $m \in (\mathbb{Z}^n)^\vee$ satisfying $\langle u_i,m\rangle=1$. By the balancing condition above, we have
\[
d_i=\sum_{u_i \sim u_j} w_{ij}\langle u_j,m \rangle.
\]
The divisor of the character $\chi^m$ in $X$ is
\[
\text{div}(\chi^m)=\sum_{j} \langle u_j,m\rangle D_j,
\]
where the sum is over all torus-invariant prime divisors in $X$ \cite[Proposition 4.1.2]{Cox-Little-Schenck}. Therefore, we have the rational equivalence
\[
-D_i \sim \sum_{j\neq i} \langle u_j,m\rangle D_j,
\]
where the sum is over all torus-invariant prime divisors in $X$ not equal to $D_i$.
Since $D_i$ and $D_j$ are disjoint when $u_iu_j$ does not generate a cone in $F$, this implies
\[
-D_i \cdot D_i \cdot \text{cl}(Z) =  \sum_{u_i \sim u_j} \langle u_j,m\rangle \Big(D_i \cdot D_j \cdot \text{cl}(Z)\Big),
\]
where the sum is over all neighbors of $u_i$ in the graph of $F$.
It follows that
\[
d_i=\sum_{u_i \sim u_j} w_{ij}\langle u_j,m \rangle = \sum_{u_i \sim u_j}  \langle u_j,m \rangle \Big(D_i \cdot D_j \cdot \text{cl}(Z)\Big)=-D_i \cdot D_i \cdot \text{cl}(Z).
\]

We now show that the tropical Laplacian of $F$ has exactly one negative eigenvalue.
Choose any resolution of singularities $\pi: \widetilde{Z} \longrightarrow Z$. By the projection formula, for any $i$ and $j$,
\[
D_i \cdot D_j \cdot \text{cl}(Z)= \pi^*\big(c_1(L_i)\big) \cup \pi^*\big(c_1(L_j)\big) \cap \text{cl}(\widetilde{Z}).
\]
Let $V$ be the real vector space with basis $e_1,e_2,\ldots$, and consider the linear map to the N\'eron-Severi space
\[
V \longrightarrow \text{NS}^1_\mathbb{R}(\widetilde{Z}), \qquad e_i \longmapsto \pi^*\big(c_1(L_i)\big).
\]
By the computation made above, the quadratic form associated to the tropical Laplacian of $F$ is obtained as the composition
\[
\xymatrix{
V \times V \ar[r] & \text{NS}^1_\mathbb{R}(\widetilde{Z}) \times \text{NS}^1_\mathbb{R}(\widetilde{Z}) \ar[r]^{\hspace{12mm} -I} & \mathbb{R},
}
\]
where $I$ is the intersection form on $\widetilde{Z}$. By the Hodge index theorem \cite[Chapter 4]{Griffiths-Harris}, $-I$ has signature of the form $(\rho-1,1,0)$, and hence the tropical Laplacian of $F$ has at most one negative eigenvalue. Since $X$ is projective, the tropical Laplacian has, in fact, exactly one negative eigenvalue.
\end{proof}

The following application of Milman's converse to the Krein-Milman theorem relates extremality to the strongly positive Hodge conjecture, 
see \cite[Proof of Proposition 5.2]{DemaillyHodge}.
 
\begin{proposition}\label{Milman'sConverse}
Let $\mathscr{T}$ be a $(p,p)$-dimensional closed current on $X$ of the form
\[
\mathscr{T}=\lim_{i \to \infty} \mathscr{T}_i, \quad \mathscr{T}_i=\sum_j \lambda_{ij} [Z_{ij}],
\]
where $\lambda_{ij}$ are nonnegative real numbers and $Z_{ij}$ are $p$-dimensional irreducible subvarieties of $X$.
If $\mathscr{T}$ generates an extremal ray of the cone of strongly positive closed currents on $X$, then
 there are nonnegative real numbers $\lambda_i$ and $p$-dimensional irreducible subvarieties $Z_i \subseteq X$  such that
\[
\mathscr{T}=\lim_{i \to \infty} \lambda_i [Z_i].
\]
\end{proposition}

\begin{proof}
For a $(p,p)$-dimensional positive current $T$ on $X$, we set
\[
||T||:=\int_X T \wedge w^p,
\]
where $w$ is the fixed K\"ahler form on $X$.
Consider the sets of positive currents
\begin{eqnarray*}
\mathscr{S}&:=&\Bigg\{ T: \text{$T$ is a $(p,p)$-dimensional positive closed current with $||T||=1$}\Bigg\},\\
\mathscr{K}&:=&\Bigg\{\frac{[Z]}{||[Z]||} : \text{$Z$ is a $p$-dimensional irreducible subvariety of $X$}\Bigg\} \subseteq \mathscr{S}.
\end{eqnarray*}
Banach-Alaoglu theorem \cite[Theorem 3.15]{Rudin} shows that $\mathscr{S}$ is compact, and hence
the closure $\overline{\mathscr{K}} \subseteq \mathscr{S}$ and the closed convex hull $\overline{\text{co}}(\mathscr{K}) \subseteq \mathscr{S}$ are compact.
Since the space of $(p,p)$-dimensional currents on $X$ is locally convex, Milman's theorem \cite[Theorem 3.25]{Rudin} applies to these compact sets:
Every extremal element of $\overline{\text{co}}(\mathscr{K})$ is contained in $\overline{\mathscr{K}}$.
To conclude, we note that 
\[
\frac{\mathscr{T}}{||\mathscr{T}||} \in \overline{\text{co}}(\mathscr{K}).
\]
Indeed, the current $\mathscr{T}_i$ is nonzero for sufficiently large $i$, and 
\[
\frac{\mathscr{T}}{||\mathscr{T}||}=\lim_{i \to \infty} \frac{\mathscr{T}_i}{||\mathscr{T}_i||}, \qquad\frac{ \mathscr{T}_i}{|| \mathscr{T}_i||} \in \text{co}(\mathscr{K}).
\]
Furthermore, since the cone of strongly positive closed currents contains $\overline{\text{co}}(\mathscr{K})$, the current $\mathscr{T}/||\mathscr{T}||$ is an extremal element of $\overline{\text{co}}(\mathscr{K})$. It follows from Milman's theorem that 
\[
\frac{\mathscr{T}}{||\mathscr{T}||} \in \overline{\mathscr{K}}.
\]
In other words, there are $p$-dimensional irreducible subvarieties $Z_i \subseteq X$ such that
\[
\frac{\mathscr{T}}{||\mathscr{T}||}=\lim_{i \to \infty} \frac{[Z_i]}{||[Z_i]||}.
\]
\end{proof}

\subsection{}

Suppose $F$ is a $2$-dimensional real weighted fan in $\mathbb{R}^n$ with the following properties:
\begin{enumerate}[--]
\item $F$ is balanced, unimodular, and non-degenerate,
\vspace{1mm}
\item $G(F)$ is connected and locally extremal,
\vspace{1mm}
\item the negative edges of $G(F)$  are pairwise disjoint and in general position.
\end{enumerate}
Let $u_1u_2,u_3u_4,\ldots$ be the negative edges of $G(F)$, and let $m$ be the number of negative edges. Since the negative edges are pairwise disjoint and in general position, we may define
\[
\widetilde{F}:=(((F^-_{12})_{34}^-)^-_{56}\ldots)^-_{2m-12m}.
\]
The resulting weighted fan $\widetilde{F}$  has the following properties:
\begin{enumerate}[--]
\item $\widetilde{F}$ is positive,
\vspace{1mm}
\item $\widetilde{F}$ is balanced, unimodular, and non-degenerate,
\vspace{1mm}
\item $G(\widetilde{F})$ is connected and locally extremal.
\end{enumerate}
In addition,  by Proposition \ref{MinusSignature}, 
\[
n_-(\widetilde{F}) \ge (\text{the number of negative edges of $G(F)$}).
\]
We construct an example of $F$ in $\mathbb{R}^4$ with the stated properties.

We start from a geometric realization $G \subseteq \mathbb{R}^4 \setminus \{0\}$ of the complete bipartite graph
\[
{\xymatrixcolsep{2pc}\xymatrixrowsep{3pc}
\xymatrix{
e_1 \ar@{-}[d] \ar@{-}[dr] \ar@{-}[drr] \ar@{-}[drrr]& e_2  \ar@{-}[d]  \ar@{-}[dl] \ar@{-}[dr] \ar@{-}[drr]& e_3  \ar@{-}[d] \ar@{-}[dll] \ar@{-}[dl] \ar@{-}[dr]& e_4  \ar@{-}[d] \ar@{-}[dlll] \ar@{-}[dll] \ar@{-}[dl]\\
f_1&f_2&f_3&f_4,
}
}
\]
where $e_1,e_2,e_3,e_4$ are the standard basis vectors of $\mathbb{R}^4$ and $f_1,f_2,f_3,f_4$ are suitable primitive integral vectors to be determined. 

Let $M$ be the matrix with row vectors $f_1,f_2,f_3,f_4$, and let $\mathscr{P}$ be the collection of cones
\[
\Big\{0\Big\} \cup \Big\{\text{cone}(u_i) \mid \text{$u_i$ is a vertex of $G$}\Big\} \cup \Big\{\text{cone}(u_iu_j) \mid \text{$u_iu_j$ is an edge of $G$}\Big\}.
\]
If the determinant of $M$ is nonzero, then $\{f_1,f_2,f_3,f_4\}$ is linearly independent, and hence $G$ is locally extremal.  

\begin{lemma}\label{GenericMinors}
 If all $2 \times 2$ minors of $M$ are nonzero, then every edge of $G$ is in general position, and $\mathscr{P}$ is a fan.
\end{lemma}

\begin{proof}
We show that every edge of $G$ is in general position. It is easy to check from this that $\mathscr{P}$ is a fan. By symmetry, it is enough to show that
\begin{eqnarray*}
\text{span}(e_1,f_1) \cap \text{span}(e_2,f_2)&=&0, \\
\text{span}(e_1,f_1) \cap \text{span}(e_1,f_2)&=&\text{span}(e_1),\\
\text{span}(e_1,f_1) \cap \text{span}(e_2,f_1)&=&\text{span}(f_1).
\end{eqnarray*}
This follows from direct computation. For example, the intersection $\text{span}(e_1,f_1) \cap \text{span}(e_2,f_2)$ is isomorphic to the kernel of the transpose of the submatrix $M_{\{1,2\},\{3,4\}}$, which is nonsingular by assumption. Therefore we have the first equality. The other two equalities can be shown in a similar way.
\end{proof}

If the determinant of $M$ is nonzero, then $G$ is locally extremal, and hence any two balanced edge-weights on $G$ are proportional.
For a randomly chosen $M$, the graph $G$ does not admit any nonzero balanced weight.

\begin{lemma}\label{BalancingK44}
If the columns of $M$ form an orthogonal basis of $\mathbb{R}^4$, then $G$ admits a nonzero balanced integral weight, unique up to a constant multiple.
\end{lemma}

\begin{proof}
The uniqueness follows from the connectedness and the local extremality of $G$. We define edge-weights on $G$ by setting
\[
\left(\begin{array}{cccc} 
\text{w}(e_1f_1) &\text{w}(e_2f_1) & \text{w}(e_3f_1) & \text{w}(e_4f_1)\\
\text{w}(e_1f_2) &\text{w}(e_2f_2) & \text{w}(e_3f_2) & \text{w}(e_4f_2)\\
\text{w}(e_1f_3) &\text{w}(e_2f_3) & \text{w}(e_3f_3) & \text{w}(e_4f_3)\\
\text{w}(e_1f_4) &\text{w}(e_2f_4) & \text{w}(e_3f_4) & \text{w}(e_4f_4)\\
\end{array}\right):=M,
\]
where $\text{w}(e_if_j)$ denote the weight on the edge $e_if_j$.
It is straightforward to check that $G$ is balanced. For example, the balancing condition for $G$ at $f_1$ is
\[
f_1= f_{11} e_1+f_{12} e_2+f_{13} e_3+f_{14}e_4,
\]
and the balancing condition for $G$ at $e_1$ is
\[
(f_{11}^2+f_{21}^2+f_{31}^2+f_{41}^2)e_1=f_{11}f_1+f_{21}f_2+f_{31}f_3+f_{41}f_4.
\]
\end{proof}

Suppose that $M$ is an integral matrix all of whose $2 \times 2$ minors are nonzero. If columns of $M$ form an orthogonal basis of $\mathbb{R}^4$, then we can construct a balanced weighted fan $F$ on $\mathscr{P}$ using Lemmas \ref{GenericMinors} and \ref{BalancingK44}.
If furthermore all the entries of $M$ are either $0$ or $\pm 1$, then $F$ is unimodular, and if each row and column of $M$ contains at most one negative entry, then the negative edges of $G(F)$  are pairwise disjoint. 

As a concrete example, we take
\[
M
:=\left(
\begin{array}{rrrr}
0&1&1&1\\
1&0&-1&1\\
1&1&0&-1\\
1&-1&1&0\\
\end{array}
\right).
\]
The determinant of $M$ is $-9$, and the $2 \times 2$ minors of $M$ are
\begin{multline*}
-1,-1,-1,-1,-1,-1,-1,-1,
-1,-1,-1,-1,-1,-1,-1,-1,-1,-1,-1,\\
+1,+1,+1,+1,+1,
+1,+1,+1,+1,+1,+1,
-2,-2,-2,
+2,+2,+2.
\end{multline*}
It follows that $\mathscr{P}$ is a unimodular fan and all edges of $G$ are in general position.
The columns of $M$ form an orthogonal basis of $\mathbb{R}^4$, and hence $\mathscr{P}$ admits a balanced integral weight as in Lemma \ref{BalancingK44}. This defines a balanced weighted unimodular fan $F$.
The abstract graph of $G(F)$ is 
\[
\xymatrixcolsep{2pc}\xymatrixrowsep{3pc}
\xymatrix{
e_1 \ar@{-}[dr] \ar@{-}[drr] \ar@{-}[drrr]& e_2\ar@{-}[dl] \ar@{-}[dr] \ar@{--}[drr]& e_3\ar@{-}[dll] \ar@{--}[dl] \ar@{-}[dr]& e_4\ar@{-}[dlll] \ar@{-}[dll] \ar@{--}[dl]\\
f_1&f_2&f_3&f_4,
}
\]
where the three edges with negative weights are denoted by dashed lines.
Since negative edges of $G(F)$ are pairwise disjoint and in general position, we can construct the positive balanced weighted fan $\widetilde{F}$ from $F$.
We order the vertices of $G(\widetilde{F})$ by
\[
+e_1,+e_2,+e_3,+e_4,+f_1,+f_2,+f_3,+f_4,-e_2,-e_3,-e_4,-f_2,-f_3,-f_4.
\]
The tropical Laplacian of $G(\widetilde{F})$ is the symmetric  matrix
\[
\tiny
L_{G(\widetilde{F})}=\left(\begin{array}{rrrrrrrrrrrrrrrrrrr} 
3&0&0&0&0&-1&-1&-1&0&0&0&0&0&0 \\\\
0&3&0&0&-1&0&-1&0&0&0&0&0&0&-1 \\\\
0&0&3&0&-1&0&0&-1&0&0&0&-1&0&0 \\\\
0&0&0&3&-1&-1&0&0&0&0&0&0&-1&0 \\\\
0&-1&-1&-1&1&0&0&0&0&0&0&0&0&0 \\\\
-1&0&0&-1&0&1&0&0&0&-1&0&0&0&0 \\\\
-1&-1&0&0&0&0&1&0&0&0&-1&0&0&0 \\\\
-1&0&-1&0&0&0&0&1&-1&0&0&0&0&0 \\\\
0&0&0&0&0&0&0&-1&0&0&0&0&0&-1 \\\\
0&0&0&0&0&-1&0&0&0&0&0&-1&0&0 \\\\
0&0&0&0&0&0&-1&0&0&0&0&0&-1&0 \\\\
0&0&-1&0&0&0&0&0&0&-1&0&0&0&0 \\\\
0&0&0&-1&0&0&0&0&0&0&-1&0&0&0 \\\\
0&-1&0&0&0&0&0&0&-1&0&0&0&0&0
\end{array}
\right),
\]
and 
\[
\Big(n_+(\widetilde{F}), \ n_-(\widetilde{F}), \ n_0(\widetilde{F})\Big) = \big(7,3,4\big).
\]
We use the weighted fan $\widetilde{F}$ to construct the strongly positive closed current $\mathscr{T}$ in Theorem \ref{Main4}.

\begin{proof}[Proof of Theorem \ref{Main4}]
There is a refinement of $\widetilde{F}$ that is compatible with a complete fan $\Sigma_1$ on $\mathbb{R}^4$; this is a general fact on extension of fans \cite[Proposition 3.15]{Gil-Sombra}.
Applying toric Chow lemma \cite[Theorem 6.1.18]{Cox-Little-Schenck} and toric resolution of singularities \cite[Theorem 11.1.9]{Cox-Little-Schenck} to $\Sigma_1$ in that order, we can construct a subdivision $\Sigma_2$ of $\Sigma_1$ that defines a smooth projective toric variety $X$. Let $\mathscr{C}$ be the refinement of $\widetilde{F}$ that is compatible with $\Sigma_2$. Since $\mathscr{C}$ is a unimodular refinement of the $2$-dimensional unimodular weighted fan $\widetilde{F}$, it is obtained from $\widetilde{F}$ by repeated application of the construction $F \longmapsto F^+_{ij}$ in Section \ref{PlusConstruction}, see \cite[Lemma 10.4.2]{Cox-Little-Schenck}. Therefore $\mathscr{C}$ is strongly extremal, and by Proposition \ref{PositiveSignature}, 
\[
n_-(\mathscr{C})=n_-(\widetilde{F})=3.
\]
Let $\mathscr{T}:=\overline{\mathscr{T}}_{\mathscr{C}}$ be the tropical current on $X$ associated to the non-degenerate weighted fan $\mathscr{C}$.  We note that
\begin{enumerate}[(1)]
\item $\overline{\mathscr{T}}_\mathscr{C}$ is strongly positive because $\mathscr{C}$ is positive,
\item $\overline{\mathscr{T}}_\mathscr{C}$ is closed because $\mathscr{C}$ is balanced (Proposition \ref{ToricClosed}),
\item $\overline{\mathscr{T}}_\mathscr{C}$ is strongly extremal because $\mathscr{C}$ is strongly extremal (Proposition \ref{ToricExtremality}).
\end{enumerate}
We show that  $\overline{\mathscr{T}}_\mathscr{C}$ is not a weak limit of the form
\[
\lim_{i \to \infty} \mathscr{T}_i, \quad \mathscr{T}_i=\sum_j \lambda_{ij} [Z_{ij}],
\]
where $\lambda_{ij}$ are nonnegative real numbers and $Z_{ij}$ are irreducible surfaces in $X$.
If otherwise, since $\overline{\mathscr{T}}_\mathscr{C}$ is strongly extremal, Proposition \ref{Milman'sConverse} implies that
 there are nonnegative real numbers $\lambda_i$ and $p$-dimensional irreducible subvarieties $Z_i \subseteq X$  such that
\[
\overline{\mathscr{T}}_\mathscr{C}=\lim_{i \to \infty} \lambda_i [Z_i].
\]
Therefore, by Proposition \ref{OneNegativeEigenvalue}, the tropical Laplacian of $\{\overline{\mathscr{T}}_\mathscr{C}\}$ has at most one negative eigenvalue.
However, by Theorem \ref{Main3}, 
\[
\{\overline{\mathscr{T}}_\mathscr{C}\}=\mathscr{C},
\]
whose tropical Laplacian has three negative eigenvalues, a contradiction.
\end{proof}


\end{document}